\documentclass[a4paper,11pt]{amsart}
\usepackage[left=2cm,top=3cm,right=2cm,bottom=2cm]{geometry}
\usepackage{amsmath}
\usepackage{amssymb}
\usepackage{amsfonts}
\usepackage{amsthm}
\usepackage{stmaryrd}
\usepackage{url}
\usepackage{hyperref}
\usepackage{fancyhdr}
\usepackage{mathrsfs}
\usepackage{mathtools}
\usepackage[demo]{graphicx}
\usepackage{color}

\usepackage{amstext}
\usepackage{array}

\usepackage{enumitem}
\setlist[enumerate]{leftmargin=25pt}
\setlist[itemize]{leftmargin=25pt}

\newtheorem{thm}{Theorem}[section]
\newtheorem{lemma}[thm]{Lemma}
\newtheorem{prop}[thm]{Proposition}
\newtheorem{cor}[thm]{Corollary}
\newtheorem{conj}[thm]{Conjecture}

\theoremstyle{definition}

\newtheorem*{rem}{Remark}
\newtheorem{claim}{Claim}

\theoremstyle{definition}
\newtheorem{defn}[thm]{Definition}

\DeclareMathOperator{\ord}{ord}

\newcommand{\rad}{\ensuremath{\operatorname{rad}}}
\newcommand{\Id}{\ensuremath{\operatorname{Id}}}

\newcommand{\bra}[1]{{\left({#1}\right)}}

\newcommand{\ent}[1]{{\left[{#1}\right]}}
\newcommand{\scal}[1]{{\left\langle{#1}\right\rangle}}
\newcommand{\set}[1]{{\left\{{#1}\right\}}}

\DeclarePairedDelimiter\abs{\lvert}{\rvert}

\newcommand{\Ga}{\ensuremath{\Gamma}}
\newcommand{\de}{\ensuremath{\delta}}

\newcommand{\eps}{\ensuremath{\varepsilon}}
\newcommand{\ze}{\ensuremath{\zeta}}
\newcommand{\ka}{\ensuremath{\kappa}}
\newcommand{\La}{\ensuremath{\Lambda}}
\newcommand{\la}{\ensuremath{\lambda}}
\newcommand{\sig}{\ensuremath{\sigma}}
\newcommand{\Sig}{\ensuremath{\Sigma}}
\newcommand{\vphi}{\ensuremath{\varphi}}
\newcommand{\om}{\ensuremath{\omega}}

\newcommand{\bears}{\begin{eqnarray*}}
\newcommand{\eears}{\end{eqnarray*}}

\newcommand{\ZZ}{\ensuremath{\mathbb{Z}}}

\newcommand{\FF}{\ensuremath{\mathbb{F}}}
\newcommand{\QQ}{\ensuremath{\mathbb{Q}}}
\newcommand{\RR}{\ensuremath{\mathbb{R}}}

\newcommand{\wh}[1]{\ensuremath{\widehat{#1}}}

\newcommand{\CC}{\ensuremath{\mathbb{C}}}

\newcommand{\Gal}{\ensuremath{\textnormal{Gal}}}

\newcommand{\Php}{\ensuremath{\Phi_p(X^{N/p})}}
\newcommand{\Phq}{\ensuremath{\Phi_q(X^{N/q})}}

\newcommand{\sm}{\ensuremath{\setminus}}
\newcommand{\ssq}{\ensuremath{\subseteq}}

\newcommand{\vn}{\ensuremath{\varnothing}}

\newcommand{\mf}{\ensuremath{\mathfrak}}
\newcommand{\ol}{\ensuremath{\overline}}
\newcommand{\wtn}[1]{\ensuremath{\bs1_{#1}\ast\bs1_{-#1}}}

\newcommand{\pdiv}{\ensuremath{\mid\!\mid}}

\newcommand{\bs}[1]{\ensuremath{\mathbf{#1}}}
\newcommand{\out}[1]{}

\numberwithin{equation}{section}

\title[Formal duality in finite cyclic groups]{Formal duality in finite cyclic groups}



\author{Romanos Diogenes Malikiosis} 
\thanks{This research project is supported by Alexander von Humboldt Foundation.}
\address{Technische Universit\"at Berlin, Institut f\"ur Mathematik,
Sekretariat MA 4-1,
Stra{\ss}e des 17. Juni 136,
D-10623 Berlin, Germany}
\email{malikios@math.tu-berlin.de}

\begin{document}
 
 \begin{abstract}
 The notion of formal duality in finite Abelian groups appeared recently in relation to spherical designs, tight sphere packings, and energy minimizing configurations in
 Euclidean spaces. For finite cyclic groups it is conjectured that there are no primitive formally dual pairs besides the trivial one and the TITO configuration. This conjecture has been
 verified for cyclic groups of prime power order, as well as of square-free order. In this paper, we will confirm the conjecture for other classes of cyclic groups,
 namely almost all cyclic groups of order a product of two prime powers, with finitely many exceptions for each pair of primes, or whose order $N$ satisfies $p\pdiv N$, where $p$ a prime
 satisfying the so-called self-conjugacy property with respect to $N$. 
 For the above proofs, various tools were needed: the \emph{field descent method}, used chiefly for the circulant Hadamard conjecture, the techniques of Coven \&
 Meyerowitz for sets that tile $\ZZ$ or $\ZZ_N$ by translations, dubbed herein as \emph{the polynomial method}, 
 as well as basic number theory of cyclotomic fields, especially the splitting of primes in a given cyclotomic extension.
 \end{abstract}

 \maketitle
 
 \section{Introduction}
 \bigskip
 
 A fundamental problem in physics is the determination of ground states in a given space, with a fixed density of particles and a pair potential. These ground states are also
 called \emph{minimal energy configurations}. A typical example is the equilibrium state of electrons in a shell of an atom. Problems of this sort are extremely difficult
 to solve rigorously; the minimal energy configuration of five points on a sphere has recently gained some notoriety, having been determined for certain special cases for the
 potential function \cite{5points}.
 
 In the Euclidean space $\RR^n$ periodic configurations of fixed density are studied, 
 say $\rho=1$; a set $\La\ssq\RR^n$ is called \emph{periodic} if it satisfies $\La+L=\La$ for some lattice $L$, and its period is $\rho(\La)=N/\det(L)$, where $\det(L)$ is the volume of a fundamental
 parallelepiped of $L$. The lattice $L$ is called the \emph{period lattice} of $\La$, and there always
 exists a maximum such lattice. The Gaussian potential function is considered in this case \cite{CKRS,CKS}, defined by $G_c(r)=e^{-\pi cr^2}$, $c\in\RR$; this is the \emph{Gaussian core model}.
 For a potential function $f$ and a periodic set $\La=\bigcup_{j=1}^N(t_j+L)$, the total energy of the system is
 \begin{equation}\label{totalenergy}
 E_f(\La)=\frac{1}{N}\sum_{i,j=1}^N\sum_{\substack{v\in L\\ v\neq0\text{ if }i=j}}f(\abs{v+t_i-t_j}).
 \end{equation}
 When the density $\rho$ is very small, or when $f=G_c$ with $c\rightarrow\infty$, the optimal configuration approaches the optimal sphere packing \cite{CKRS,CKS}.
 
 Apart from the 1-dimensional case, where the energy minimizing configuration of density $1$ is $\ZZ$, there are no proofs that certain structured
 configurations minimize energy, however, there is strong numerical evidence towards certain patterns. In the study conducted in \cite{CKS}
 for the Gaussian core model and varying densities,
 a remarkable sort of symmetry was revealed between optimal configurations in densities $\rho$ and $1/\rho$, the so-called \emph{formal duality}. In particular, when $n\leq9$,
 the optimal configurations for densities $\rho$ and $1/\rho$ are either dual lattices, or when they are not lattices, they are formally dual periodic sets. Formally dual sets
 satisfy a very strong property, namely a generalization of the Poisson summation formula. We note that in all of the above situations, the density
 is considered fixed. For systems in equilibrium without the fixed density restriction this task is even more difficult; we refer the reader to \cite{particles}.
 
 We remind that for a Schwartz function $f:\RR^n\rightarrow\CC$, and a lattice $\La\ssq\RR^n$, the Poisson summation formula states that
 \begin{equation}\label{Poisson}
  \sum_{x\in\La}f(x)=\frac{1}{\det(\La)}\sum_{y\in\La^{\star}}\hat{f}(y)
 \end{equation}
 where $\hat{f}$ denotes the Fourier transform of $f$, defined by
 \begin{equation}\label{Fourier}
 \hat{f}(y)=\int f(x)e^{-2\pi i\scal{x,y}}dx,
 \end{equation}
 and $\La^{\star}:=\set{x\in\RR^n:\scal{x,y}\in\ZZ,\forall y\in\La}$ is the dual lattice of $\La$; finally, $\det(\La)$ is the volume of any fundamental parallelepiped of $\La$,
 also known as the covolume of $\La$. A consequence of this formula is 
 \begin{equation}\label{energies}
  f(0)+E_f(\La)=\frac{1}{\det(\La^{\star})}(\hat{f}(0)+E_{\hat{f}}(\La^{\star}))
 \end{equation}
 hence the energy $E_{\hat{f}}(\La^{\star})$ is a real number when $f$ is a real function, despite the fact that $\hat{f}$ is not in general real.
 Moreover, a lattice $\La$ minimizes $E_f$ among periodic configurations of density $\rho=1/\det(\La)$ if and only
 if $\La^{\star}$ minimizes $E_{\hat{f}}$ among periodic configurations of density $1/\rho$. The significance of the Gaussian core model is then justified by the relation $\wh{G}_c=c^{-n/2}G_{1/c}$
 (in particular, $\wh{G}_1=G_1$).
  
 It is known that \eqref{Poisson} characterizes dual pairs of lattices \cite{Cordoba} among discrete subsets of $\RR^n$; consider the \emph{average pair sum} for the periodic
 configuration $\La=\bigcup_{j=1}^N(t_j+L)$
 \begin{equation}\label{genPoisson}
  \Sig_f(\La)=\frac{1}{N}\sum_{i,j=1}^N\sum_{v\in L}f(v+t_i-t_j),
 \end{equation}
 or simply put, $\sum_f(\La)=f(0)+E_f(\La)$. We note that the above expression is independent of the choice of $t_i$, as the inner sum ranges over all elements of the lattice $\La$, hence the argument $v+t_i-t_j$
 takes each value from the translated lattice $L+t_i-t_j$ exactly once, and $L+t_i-t_j=L+t'_i-t'_j$ holds, for all $t'_i\in t_i+L$, $t'_j\in t_j+L$.
 
 Another consequence of \eqref{Poisson} is $\sum_f(\La)=\rho(\La)\sum_{\hat{f}}(\La^{\star})$ when $\La$ is a lattice (i.e. $N=1$). Minimal energy periodic
 configurations for the Gaussian core model found in \cite{CKS} at densities $\rho$ and $1/\rho$, say $\La$ and $\Ga$ respectively, were proven to satisfy
 \begin{equation}\label{average}
  \Sig_f(\La)=\rho(\La)\Sig_{\hat{f}}(\Ga).
 \end{equation}
 \begin{defn}\label{maindefn}
  Two periodic sets $\La,\Ga\ssq\RR^n$ are called formally dual if they satisfy \eqref{average} for every Schwartz function $f$.
 \end{defn} 
 %Thus, the determination of formally dual $\La$ and $\Ga$ is in order.

 Formally dual pairs that are not lattices seem to appear in a great scarcity in the 1-dimensional case; the only known example is $2\ZZ\cup(2\ZZ+\frac{1}{2})$ (or a scaled version
 thereof), the so-called \emph{TITO configuration}\footnote{TITO stands for two in-two out.}. 
 While there are more high dimensional examples of formal duality, they do not seem to appear yet in numerical computations; in the computations performed in \cite{CKS} and
 Coulangeon-Sch\"urmann\footnote{Private communication.} for $n\leq9$, all optimal configurations are linear images of $\ZZ^n$, TITO$\times\ZZ^{n-1}$, or TITO$^2\times\ZZ^{n-2}$. 
 Thus, the characterization of all 1-dimensional formally dual sets is in order; it was conjectured in \cite{CKRS}, 
 that $\ZZ$ and TITO are the only discrete periodic subsets of $\RR$ with density $1$, possessing a formal dual set.
 
 The above can be rephrased in terms of cyclic groups (for more details, see \cite{CKRS}). Let $\ZZ_N$ denote the cyclic groups of $N$ elements, and call a subset $T\ssq\ZZ_N$ \emph{primitive}, 
 if it is not contained in any proper coset of $\ZZ_N$. Then the aforementioned conjecture is equivalent to the following: 
 \begin{conj}\label{mainconj}
  The only primitive subsets of $\ZZ_N$ possessing a formal dual subset are
 $\set{0}\ssq\ZZ/\ZZ$ and $\set{0,1}\ssq\ZZ_4$.
 \end{conj}
 This conjecture has been verified when $N$ is a prime power, by Sch\"uler \cite{Schueler16} for $p$ odd or when $N$ an even
 power of $2$, and by Xia, Park, and Cohn \cite{XPC} in the remaining cases, as well as when $N$ is square-free. 
 
 We briefly mention that there is an infinite family of primitive formally dual sets in non--cyclic groups, the \emph{Gauss sum configurations}. The set $T=\set{(n,n^2):n\in\ZZ_p}\ssq\ZZ_p\times\ZZ_p$
 is primitive formally self-dual \cite{CKRS}. This example has been generalized to the groups $\ZZ_{p^k}\times\ZZ_{p^k}$ \cite{XPC}. It might not be coincidental that this is a \emph{Sidon set} (see Exercise
 2.2.7 in \cite{taovu2006additive}); as we will see in the cyclic case, it seems that such sets can only exist if their sizes both equal $\sqrt{N}$, where $N$ is the order of the group, and their sets of
 differences spread out in the group $G$. We prove in particular that when $N$ is divisible by at most two primes, then every element in $\ZZ_N^{\star}$ (which consists of a large part of the group $\ZZ_N$)
 appears exactly once as difference of the form $t-t'$, where $t,t'\in T$ and $T\ssq\ZZ_N$ primitive, possessing a formally dual set.

 In this paper, we provide ample evidence towards the veracity of Conjecture \ref{mainconj}. In particular, we prove that the conjecture is true when:
 \begin{enumerate}
  \item $N=p^2q^2$, for $p$, $q$ distinct primes.
  \item $N=p^mq^n$, for $p$, $q$ distinct primes, with possibly finitely many exceptions for each pair $(p,q)$.
  \item A prime $p$ divides \emph{exactly} $N$, that is, $p\mid N$, but $p^2\nmid N$, and $p$ is \emph{self-conjugate} $\bmod N$, i.e. there exists $j\in\ZZ$ such that $p^j\equiv-1\bmod N$.
 \end{enumerate}
 %The latter case implies that the set of orders $N$ for which the conjecture is confirmed, has density $1$.
 
 We should note that tools from different areas were introduced in order to tackle these cases; for case (1), the ideas of Coven-Meyerowitz \cite{CM} for sets that tile $\ZZ$ by translations
 were used; for (2), the so-called \emph{field descent method} was used, that was developed by Schmidt \cite{Schmidt99,SchmidtFDM} chiefly for the circulant Hadamard conjecture, 
 as well as for applications in combinatorial designs and coding theory; for (3), the arithmetic of cyclotomic fields, especially the splitting of primes in cyclotomic extensions.
 
 We will try to keep this paper as self-contained as possible; it is organized as follows: in Section \ref{basicnotation}, we provide the necessary number theoretic background to the problem.
 In Section \ref{polynomial}, we develop the ``polynomial method'', and in Section \ref{structural} we prove basic results with respect to formal duality in cyclic groups.
 The field descent method is introduced in Section \ref{fielddescent}, and in Section \ref{oneprime}, we re-prove the prime power case, showcasing the importance of the new ideas
 involved. In Section \ref{twoprimes}, we apply the field descent method, as well as the polynomial method and use them to prove the conjecture for products of two prime powers, except for finitely many cases
 for every pair $(p,q)$. In Section \ref{beyond}, we prove the conjecture when $p\pdiv N$ and $p$ self--conjugate $\bmod N$. 
 In the appendix we provide some numerical data for case (2) above that indicate how few exceptions
 for each pair $(p,q)$ exist.

 \bigskip
 \section{Basic number theoretic background}\label{basicnotation}%%%%%MAKE IT A PROPER SECTION
 \bigskip
 
 \subsection{Notation}
 Throughout this paper, we will denote by $\ZZ_N$ the ring of integers modulo $N$, which as an additive group is cyclic of order $N$. We also denote $\ze_N=e^{2\pi i/N}$, a 
 primitive $N$th root of unity. For a function $f:\ZZ_N\rightarrow\CC$, the Fourier
 transform is defined as 
 \[\bs{F}f(y)=\hat{f}(y)=\sum_{x\in\ZZ_N}f(x)\ze_N^{xy}.\]
 For a set $T\ssq\ZZ_N$ we denote by $\bs1_T$ its indicator function.
 Some usual arithmetic functions will be needed here: the number of distinct prime factors of $N$ will be denoted by $\om(N)$ and the product of the distinct prime factors of $N$
 as $\rad(N)$, which is also known as the \emph{radical} of $N$. Then, we have the M\"{o}bius function $\mu(n)$, defined as
 \[\mu(n)=\begin{cases}
           (-1)^{\om(N)},\	\	&\text{ if }N\text{ is square--free}\\
           0,\	\	&\text{ otherwise.}
          \end{cases}
\]
We also set $\Id(N)=N$ and $\bs1(N)=1$ for all $N$, and $e(N)=1$ if $N=1$, while $e(N)=0$ otherwise.
$\vphi(N)$ is the usual \emph{Euler totient function}, which enumerates the positive integers prime to $N$ that are $\leq N$. The following well-known formula holds
\[\vphi(N)=N\prod_{\substack{p\mid N\\p\text{ prime}}}\bra{1-\frac{1}{p}}.\]
Finally, for a prime $p$ we define $\nu_p(N)$ by $p^{\nu_p(N)}\pdiv N$.
The symbol $\ast$ will denote convolution, either additive or multiplicative (i.e. \emph{Dirichlet convolution}), depending on the context. For the classical arithmetic functions,
it will always be multiplicative; for example, the following formulae hold
\[\vphi=\mu\ast\Id,\	\	\Id=\vphi\ast\bs1,\]
an example of \emph{M\"{o}bius inversion}. We also have $f\ast e=f$ for all $f$, that is $e$ is the identity element with respect to the Dirichlet convolution, and
$\bs1\ast\mu=e$, that is, $\mu$ is the inverse of $\bs1$. For these basic facts on arithmetic functions we refer the reader to \cite{Rose}, or any other book on basic number theory.
 
 \subsection{Cyclotomic fields}
 We list some of the basic results on cyclotomic fields, mainly the splitting of primes in cyclotomic extensions of $\QQ$. For these basic facts we refer the reader to
 any of \cite{Marcus,Neukirch,Washington}.
 
 Cyclotomic fields have the form $\QQ(\ze_N)$; we remind that if $N\equiv2\bmod 4$ then $\QQ(\ze_N)=\QQ(\ze_{2N})$, which is also recovered from the fact that the degree of the extension
 satisfies $[\QQ(\ze_N):\QQ)]=\vphi(N)$. This extension is always an Abelian extension, that is, it is Galois with Abelian Galois group. In particular,
 \[\Gal(\QQ(\ze_N)/\QQ)\cong\ZZ_N^{\star},\]
 and the canonical group isomorphism is defined by $g\mapsto \sig_g$ for every $g\in\ZZ_N^{\star}$, where $\sig_g(\ze_N)=\ze_N^g$.
 The ring of integers of the field $\QQ(\ze_N)$ is $\ZZ[\ze_N]$, and every ideal can be factorized uniquely into a product of prime ideals (which are also maximal, as algebraic number rings are
 Dedekind domains). The most important fact for our purposes is the splitting of the ideal $p\ZZ[\ze_N]$ into primes, where $p$ is a (rational) prime.
 
 \begin{thm}\label{splitting}
  Let $N$ be a positive integer, $p$ be a prime and $m$ the $p$-free part of $N$, i.e. $N=p^am$, where $p\nmid m$. Then
  \[p\ZZ[\ze_N]=(\mf{P}_1\mf{P}_2\dotsm\mf{P}_r)^e,\]
  where $\mf{P}_1,\dotsc,\mf{P}_r$ distinct prime ideals of $\ZZ[\ze_N]$, $e=\vphi(p^a)$ the ramification index, and $r=\vphi(m)/f$, where $f$ is the multiplicative order of $p\bmod m$, that is,
  $p^f\equiv1\bmod m$ and $f$ is the smallest positive integer with this property (also called the inertia degree). Furthermore, if we define $\ka(\mf{P})=\ZZ[\ze_N]/\mf{P}$ (the residue field), 
  so that $\ka(p\ZZ)=\ZZ/p\ZZ$, we have $f=[\ka(\mf{P}_i):\ZZ/p\ZZ]$, in other words, the inertia degree is the degree of the residue field extension.
 \end{thm}
 
 \begin{cor}\label{psplitting}
  With the previous notation, if $N=p^a$ (i.e. $m=1$), then $r=f=1$, and
  \[p\ZZ[\ze_N]=\mf{P}^{\vphi(N)}.\]
  The ideal $\mf{P}$ is principal, and is generated by $1-\ze$, where $\ze$ is any primitive $N$th root of unity.
 \end{cor}
 
 The prime ideals $\mf{P}_i$ are said to lie above $p$; the Galois group $G=\Gal(\QQ(\ze_N)/\QQ)$ acts transitively on those. For $\mf{P}\mid p$, the subgroup
 \[G_{\mf{P}}=\set{\sig\in G:\sig\mf{P}=\mf{P}}\]
 is called the \emph{decomposition group} of $\mf{P}$. Since $G$ is Abelian, $G_{\mf{P}}$ is the same for all $\mf{P}\mid p$ (in general, these groups are conjugate with each other).
 
 Corollary \ref{psplitting} shows that $1-\ze$ is not a unit in $\ZZ[\ze_N]$, when $N$ is a power of a prime and $\ze$ is a primitive $N$th root of unity, otherwise $(1-\ze)\ZZ[\ze_N]=\ZZ[\ze_N]$.
 This can be seen by taking the value of the cyclotomic polynomial $\Phi_N(X)$ at $X=1$:
 \[\Phi_N(1)=\prod_{\gcd(g,N)=1}(1-\ze_N^g).\]
 As
 \begin{equation}\label{Phivalues}
  \Phi_N(1)=\begin{cases}
             p,\	\	&\text{ if }N\text{ is a power of }p\\
             1,\	\	&\text{ otherwise},
            \end{cases}
 \end{equation}
 we obtain the following Lemma.
 
 \begin{lemma}
  Let $\ze$ be a primitive $N$th root of unity. $1-\ze$ is a unit in $\ZZ[\ze_N]$ if and only if $N$ is not a prime power.
 \end{lemma}

 %%%%VANISHING SUMS OF ROOTS OF UNITY

 %%%%%INCLUDE ZEROS OF T(X), DERIVATIVES, X^{-1}
 
 \bigskip
 \section{The polynomial method}\label{polynomial}
 \bigskip
 
 With every multiset $T$ with elements from $G$ we associate an element of the group ring $\ZZ[G]$, 
 namely $\sum_{g\in G}m_g g$, where $m_g$ is the multiplicity of $g$ in $T$. When $G$ is cyclic, we can write instead $\sum_{g\in G}m_g X^g$, the so-called \emph{mask polynomial}, 
 which is an element of $\ZZ[X]/(X^N-1)\cong\ZZ[G]$, where $N=\abs{G}$. Both notations have appeared in bibliography before; see for example \cite{KaSchmidt,Schmidt99,SchmidtFDM}
 for the group ring notation, or \cite{CM} for the mask polynomial. The former has been advantageous in algebraic coding theory, while the latter in tiling problems
 on $\ZZ$ or the finite cyclic groups, $\ZZ_N$ \cite{MK}. %%%%CITATIONS
 
 In this paper, we will adhere to the polynomial notation; the mask polynomial of the multiset $T$ will be denoted simply by $T(X)$. Now let $d\mid n$, and define $d\cdot T$ to be the multiset
 of elements $dt$ for $t\in T$, \emph{counting multiplicities}. For example, if $T=\set{0,2}\ssq\ZZ_4$, then $2\cdot T=\set{0,0}$, i.e. $0$ has multiplicity $2$ in $2\cdot T$. A fundamental
 observation is:
 \begin{prop}\label{maskdT}
  Let $T$ be a multiset with elements from $\ZZ_N$ and $d\mid N$. Then, $T(X^d)$ is the mask polynomial of the multiset $d\cdot T$.
 \end{prop}
 \begin{proof}
  Let $T(X)=\sum_{g\in\ZZ_N}m_gX^g$, where $m_g$ is the multiplicity of $g$, as before. Then,
  \[T(X^d)\equiv\sum_{g\in\ZZ_N}m_gX^{dg}\equiv\sum_{h\in\ZZ_N}\biggl(\sum_{dg\equiv h\bmod N} m_g\biggr)X^h\bmod(X^N-1)\]
  proving the desired fact.
 \end{proof}

 Formal duality induces polynomial congruences $\bmod (X^N-1)$, as we will see in the next section. The following Lemma is then used to show that such a congruence cannot hold, as the values of the
 derivatives of the polynomials under question on roots of unity do not agree, thus proving the non-existence of primitive formally dual sets.
 
 \begin{lemma}\label{deriv}
  Let $P(X),Q(X)\in\ZZ[X]$ such that $P(X)\equiv Q(X)\bmod(X^N-1)$, where $N$ is a positive integer. Then, for every $N$th root of unity $\ze$ (not necessarily primitive) we have
  $P(\ze)=Q(\ze)$ and $P'(\ze)\equiv Q'(\ze)\bmod N\ZZ[\ze_N]$.
 \end{lemma}
 
 \begin{proof}
  This almost follows from definition; let $R(X)\in\ZZ[X]$ be such that $P(X)-Q(X)=(X^N-1)R(X)$. From this, we readily have $P(\ze)=Q(\ze)$, when $\ze^N=1$. Differentiating both
  sides, we obtain $P'(X)-Q'(X)=(X^N-1)R'(X)+NX^{N-1}R(X)$, so substituting $X=\ze$ gives $P'(\ze)\equiv Q'(\ze)\bmod N\ZZ[\ze_N]$.
 \end{proof}
 
 The following polynomials will appear a lot in the sequel.
 
 \begin{prop}\label{deriv2}
  Let $d\mid N$ be positive integers. Consider the function
  \[F(X)=\sum_{k=0}^{N/d-1}X^{dk},\]
  which we will also write formally as $\frac{X^N-1}{X^d-1}$, even at points where $X^d-1=0$. Then, 
  \[F'(\ze)\equiv\begin{cases}
                  0\bmod N\ZZ[\ze_N],\	\	&\text{ if }\ze^d=1\text{ and }N/d\text{ is odd,}\\
                  \frac{N}{2}\ze^{-1}\bmod N\ZZ[\ze_N],\	\	&\text{ if }\ze^d=1\text{ and }N/d\text{ is even}\\
                  \frac{N\ze^{-1}}{\ze^d-1}\bmod N\ZZ[\ze_N],\	\	&\text{ otherwise.}
                 \end{cases}
\]
 \end{prop}
 
 \begin{proof}
  Obviously 
  \[F'(X)=d\sum_{k=1}^{N/d-1}kX^{dk-1},\]
  hence for $\ze^d=1$,
  \[F'(\ze)=d\ze^{-1}\frac{1}{2}\cdot\frac{N}{d}\bra{\frac{N}{d}-1}=\ze^{-1}\frac{N}{2}\bra{\frac{N}{d}-1},\]
  and the result follows easily when $\ze^d=1$. In all other cases, we use
  \[F'(X)=\bra{\frac{X^N-1}{X^d-1}}'=-\frac{dX^{d-1}(X^N-1)}{(X^d-1)^2}+\frac{NX^{N-1}}{X^d-1},\]
  whence the case $\ze^d\neq1$, $\ze^N=1$, follows.
 \end{proof}
 
 The mask polynomial of $\ZZ_N^{\star}$ will be denoted by $R_N(X)$, and appears prominently in the polynomial congruences induced by formal duality. By definition,
 \[R_N(X)=\sum_{\substack{1\leq g\leq N\\ \gcd(g,N)=1}}X^g,\]
 and the values of $R_N$ at $N$th roots of unity are the \emph{Ramanujan sums}, denoted by
 \[C_N(d)=R_N(\ze_N^d).\]
 As Ramanujan proved \cite{Ramanujansums} (see also \cite{HardyWright}), these sums are integers, and their values are given by the following formula
 \begin{equation}\label{Rsum}
 C_N(d)=\sum_{g\mid\gcd(d,N)}\mu(\frac{N}{g})g.
 \end{equation}
 These sums will appear a lot when we apply Proposition \ref{basic}, so we will also need the following \cite{HardyWright} ($d$, $N$ are integers).
 %\begin{lemma}\label{calcrsums}
 % Let $d$, $N$ be integers. Then
  \begin{equation}\label{calcrsums}
   C_N(d)=\mu\bra{\frac{N}{\gcd(d,N)}}\frac{\vphi(N)}{\vphi\bra{\frac{N}{\gcd(d,N)}}}.
  \end{equation}
 %\end{lemma}
 
 %\begin{proof}
 % First of all we note that $C_N(d)=C_N(\gcd(d,N))$, for all $d$, $N$. Indeed, by \eqref{Rsum}, $C_N(d)\in\ZZ$, so it remains invariant under the action of the Galois
 % group $\Gal(\QQ(\ze_N)/\QQ)\cong\ZZ_N^{\star}$. In particular, when we consider $\sig\in\Gal(\QQ(\ze_N)/\QQ)$ satisfying $\sig(\ze_N^d)=\ze_N^{\gcd(d,N)}$, we get
 % \[C_N(d)=\sig(C_N(d))=C_N(\gcd(d,N)).\]
 % For the rest of the proof, we will assume that $d\mid N$
 
 % Assume that $N/d$ is not square--free, say $p^2\mid\frac{N}{d}$, for some prime $p$. Then, for every $g$ dividing $d$, $N/g$ is divisible by $p^2$, or
 % $\mu(N/g)=0$, whence by \eqref{Rsum}, $C_N(d)=0$. If $N=d$, then obviously $C_N(d)=\vphi(N)$, agreeing with the right hand side of the given equation.
  
 % Suppose that $N/d=p_1p_2\dotsm p_k$ and $\rad(N)=p_1\dotsm p_kp_{k+1}\dotsm p_m$, where the $p_i$ are distinct primes. 
 % For $g\mid d$, if $\mu(N/g)\neq0$, then $N/g$ is square--free and $\mu(N/g)=\mu(N/d)\mu(d/g)$; this happens
 % when $\frac{N}{\rad(N)}\mid g$, so by \eqref{Rsum}
 % \begin{eqnarray*}
 % C_N(d) &=& \mu(N/d)\sum_{\frac{N}{\rad(N)}\mid g\mid d}\mu(d/g)g=\mu(N/d)\frac{N}{\rad(N)}\sum_{g'\mid\frac{d\rad(N)}{N}}\mu(\frac{d\rad(N)}{Ng'})g'\\
 % &=& \mu(N/d)\frac{N}{\rad(N)}\vphi(\frac{d\rad(N)}{N})=\mu(N/d)\frac{\vphi(N)}{\vphi(\rad(N))}\vphi(p_{k+1}\dotsm p_m)\\
 % &=& \mu(N/d)\frac{\vphi(N)}{\vphi(N/d)},
 % \end{eqnarray*}
 % as desired; we used the well known relation $\vphi=\mu\ast\text{Id}$, where $\ast$ denotes Dirichlet convolution.
 %\end{proof}
 
 \begin{lemma}\label{Rderivatives}
  Consider the polynomial $R_N(X)$, the mask polynomial of $\ZZ_N^{\star}$. Let $d\mid N$ and $\ze$ a primitive $d$th root of unity. Then
  \begin{equation}\label{Rderiv1}
  R'_N(\ze)\equiv N\ze^{-1}\sum_{g\mid N, d\nmid g}\frac{\mu(g)}{\ze^g-1}\bmod N\ZZ[\ze_N],
  \end{equation}
  unless $4\mid N$ and $d=\rad(N)$ or $2\pdiv N$ and $d=\rad(\frac{N}{2})$, in which cases
  \begin{equation}\label{Rderiv2}
  R'_N(\ze)\equiv N\ze^{-1}\ent{\frac{1}{2}+\sum_{g\mid N, d\nmid g}\frac{\mu(g)}{\ze^g-1}}\bmod N\ZZ[\ze_N]
  \end{equation}
  holds.
 \end{lemma}
 
 \begin{proof}
  The polynomial $R_N(X)$ is a sum of polynomials of the same form as $F(X)$ in Proposition \ref{deriv2}, for various $d\mid N$. In particular, we may formally write
  \[R_N(X)=\sum_{g\mid N}\mu(g)\frac{X^N-1}{X^g-1}.\]
  Let $d\mid N$ and $\ze^d=1$. Assume first that $N$ is odd; if $d\mid g$, then also $\ze^g=1$ and the derivative of $\frac{X^N-1}{X^g-1}$ at $X=\ze$ is $0$ by Proposition \ref{deriv2}, so the only terms
  that will appear in $R'_N(\ze)$ are those satisfying $d\nmid g$, proving that \eqref{Rderiv1} holds.
  
  Next, assume that $N$ is even. Again, by Proposition \ref{deriv2} the contribution of the terms satisfying $d\nmid g$ to $R'_N(\ze)$ is precisely $\mu(g)\frac{N\ze^{-1}}{X^g-1}\bmod N\ZZ[\ze_N]$.
  So, assume that $d\mid g$, so that $\ze^g=1$; without loss of generality, both $d$ and $g$ are square--free, otherwise $\mu(g)=0$ and the contribution is also $0$ anyway. 
  If $4|N$, then $N/g$ is always even for every square--free $g$, so if $d\neq\rad(N)$ there is precisely an even number of square--free $g\mid N$ for which $d\mid g$ (equal to the number of divisors of
  $\frac{\rad(N)}{d}$ when $d\mid \rad(N)$ and $0$ otherwise), hence their total contribution to $R'_N(\ze)$ is $0$ by Proposition \ref{deriv2} and \eqref{Rderiv1} holds. 
  If, on the other hand, $d=\rad(N)$, there is only one such contribution of the form $\frac{N}{2}\ze^{-1}\bmod N\ZZ[\ze_N]$ by Proposition \ref{deriv2} (namely, from $g=\rad(N)$), hence \eqref{Rderiv2} holds.
  If $2\mid N$ and the square--free $g\mid N$ is even, then the contribution is $0$ to $R'_N(\ze)$ by Proposition \ref{deriv2}. So, let $g$ be odd with $d\mid g$. Unless $d=\rad(\frac{N}{2})$, there is an 
  even number of odd square--free divisors $g$ divisible by $d$, so their total contribution is $0$ and \eqref{Rderiv1} holds. When $d=\rad(\frac{N}{2})$, the only contribution of 
  $\frac{N}{2}\ze^{-1}\bmod\ZZ[\ze_N]$ comes from $g=d$, hence \eqref{Rderiv2} holds, concluding the proof.
 \end{proof}
 
 \begin{cor}
  Let $H(X)=R_{N/d}(X^d)$, where $d\mid N$. Then, if $\ze^N=1$
  \[H'(\ze)\equiv N\ze^{-1}\ent{\frac{\eps}{2}+\sum_{g\mid N, \de\nmid g}\frac{\mu(g)}{\ze^{dg-1}}}\bmod N\ZZ[\ze_N]\]
  holds, where $\ze^d$ is a primitive $\de$th root of unity, and
  \[\eps=\begin{cases}
          1,\	\	&\text{ if }4\mid \frac{N}{d}\text{ and }\de=\rad(N/d)\text{ or }2\pdiv \frac{N}{d}\text{ and }\de=\rad(N/2d)\\
          0,\	\	&\text{ otherwise.}
         \end{cases}
\]
 \end{cor}

 When $T(X)$ vanishes on a certain $N$th root of unity, we get some information about the structure of $T\ssq\ZZ_N$. This follows from a theorem on the vanishing sums of roots of unity,
 independently proven by R\'edei \cite{Redei50,Redei54}, de Bruijn \cite{deB53} and Schoenberg \cite{Sch64}. When we consider vanishing sums of $N$th roots of unity where $N$ has at most two prime
 divisors (which are most of the cases that we consider in this paper), there is a stronger result by Lam and Leung \cite{LL}. We summarize this in the following theorem.
 
 \begin{thm}\label{vanishingsums}
  Suppose that $\sum_{j=0}^{N-1}c_j\ze_N^j=0$ for some integers $c_j$, $0\leq j\leq N-1$. If we consider $\ze_N$ formally as an element of $\ZZ[G]$, where $G=\scal{\ze_N}$, 
  then $\sum_{j=0}^{N-1}c_j\ze_N^j$ is equal to an integer linear combination of terms of the form
  \begin{equation}\label{pcycle}
  \ze_N^k(1+\ze_p+\ze_p^2+\dotsb+\ze_p^{p-1}),
  \end{equation}
  for $p$ prime divisor of $N$ and $k$ integer (these terms are called $p$-cycles). If $N$ has at most two distinct prime divisors and $c_j\geq0$ for all $j$, then we can write $\sum_{j=0}^{N-1}c_j\ze_N^j$
  as a nonnegative linear combination of terms such as the above.
 \end{thm}
 
 The second part of this theorem does not hold when $N$ has at least three distinct prime divisors, as is evident from the example
 \[(\ze_p+\dotsb+\ze_p^{p-1})(\ze_q+\dotsb+\ze_q^{q-1})+(\ze_r+\dotsb+\ze_r^{r-1})=(-1)(-1)+(-1)=0,\]
 which cannot be written as a nonnegative linear combination of $p$-, $q$- or $r$-cycles \cite{LL}.
 
 Expressed with the polynomial notation, Theorem \ref{vanishingsums} gives the results below; in the two prime case we include an additional fact that will be useful for our purposes (see also \cite{MK}):
 
 \begin{thm}\label{pqrcycles}
  Let $T\ssq\ZZ_N$, such that $T(\ze_N^d)=0$ for some $d\mid N$, where $p_1,\dotsc,p_k$ are the distinct prime divisors of $N$. Then
  \[T(X^d)\equiv\sum_{j=1}^k P_j(X^d)\Phi_N(X^{N/p_j})\bmod(X^N-1),\]
  for some $P_j\in\ZZ[X]$.
 \end{thm}
 \begin{proof}
  We simply remark that $X^k\Php$ is the mask polynomial that corresponds to the term in \eqref{pcycle} under the canonical ring isomorphism $\ZZ[X]/(X^N-1)\cong\ZZ[G]$ with $G=\scal{\ze_N}$, 
  which identifies $X$ with $\ze_N$. The rest follows from Theorem \ref{vanishingsums}.
 \end{proof}
 
 We emphasize that the polynomials $P_j$ are not unique.
 
 \begin{thm}\label{pqcycles}
  Let $T\ssq\ZZ_N$, such that $T(\ze_N^d)=0$ for some $d\mid N$, and $N/d$ has at most two distinct prime divisors, say $p$, $q$. Then,
  \begin{equation}\label{pqroots}
  T(X^d)\equiv P(X^d)\Php+Q(X^d)\Phq\bmod(X^N-1).
  \end{equation}
  The polynomials $P,Q\in\ZZ[X]$ can be taken with nonnegative coefficients. If for some integer $a>0$ we have $dp^a\mid N$ (resp. $dq^a\mid N$) and $T(\ze_N^{dp^a})\neq0$ (resp. $T(\ze_N^{dq^a})\neq0$), then
  $P\not\equiv0$ (resp. $Q\not\equiv0$) for any selection of $Q$ (resp. $P$).
 \end{thm}
 \begin{proof}
  The nonnegativity of $P$ and $Q$ follows from the second part of Theorem \ref{vanishingsums}. If $T(X^{dp^a})\neq0$, then replacing $X$ by $X^{p^a}$ we obtain
  \[T(X^{dp^a})\equiv pP(X^{dp^a})+Q(X^{dp^a})\Phq\bmod(X^N-1),\]
  and substituting $X$ by $\ze_N$, $T(\ze_N^{dp^a})=pP(\ze_N^{dp^a})\neq0$, therefore $P\not\equiv0$, as desired.
 \end{proof}
 
 If $N/d$ is only divided by one prime factor, $p$, it is understood $Q\equiv0$ at \eqref{pqroots}.

 \bigskip
 \section{Structural results on formal duality}\label{structural}
 \bigskip
 
 %Let $L\ssq\RR^n$ be a lattice, i.e. the set of all integer linear combinations of a basis of $\RR^n$, and consider $N$ points $t_1,\dotsc,t_N$, such that $t_j-t_i\notin L$ whenever
 %$i\neq j$. Then, the set
 %\[\La=\bigcup_{i=1}^N(t_i+L)\]
 %is called a $N$-periodic set; we observe that $\La-L=\La$, i.e. $\La$ is periodic with respect to $L$. When $L$ is the largest lattice with this property
 %(which is what will usually be considered), it is called the \emph{period lattice} of $\La$. 
 
 %It is known that lattices are characterized by the Poisson summation formula; we remind that if $\La$ is a lattice, $\La^{\star}:=\set{x\in\RR^n:\scal{x,y}\in\ZZ,\forall y\in\La}$
 %its dual, and $\hat{f}$ the Fourier transform of $f$ defined by
 %\[\hat{f}(y)=\int f(x)e^{-2\pi i\scal{x,y}}dx,\]
 %then the Poisson summation formula states that
 %\[\sum_{x\in \La}f(x)=\frac{1}{\det(\La)}\sum_{y\in\La^{\star}}\hat{f}(y),\]
 %for all Schwartz functions $f$. The above formula holds for every Schwartz function $f$, if and only if $\La$ and $\La^{\star}$ is a pair of dual lattices. A consequence of this is
 
 %Consider next another $M$-periodic set, say 
 %\[\Ga=\bigcup_{j=1}^M(s_j+K),\]
 %for some lattice $K$, satisfying similar properties, that is, $K$ is the period lattice of $\Ga$ and $s_j-s_i\notin K$ whenever $j\neq i$.
 
 Let $G$ be a finite Abelian group, and $\hat{G}$ its dual. %; we will denote the image of $y\in G$ under this map by $\chi_y$. With this notation, we have:
 
 \begin{defn}
  The sets $S\ssq G$, $T\ssq\hat{G}$ are formally dual if they satisfy
  \begin{equation}\label{formdualdefn}
  \abs*{\frac{1}{\abs{S}}\sum_{x\in S}\xi(x)}^2=\frac{1}{\abs{T}}\nu_T(\xi),
  \end{equation}
  for every $\xi\in \hat{G}$, where $\nu_T$ is the \emph{weight enumerator} of $T$, defined by
  \[\nu_T(\xi)=\#\set{(t,t')\in T\times T : t-t'=\xi}.\]
 \end{defn}

 Under the notation introduced in Section \ref{basicnotation}, we observe that we can rewrite the weight enumerator function simply as a convolution:
 \begin{equation}
  \nu_T(\xi)=\wtn{T}(\xi).
 \end{equation}
 From now on, we will focus on cyclic groups, and 
 we will use the explicit isomorphism between $\ZZ_N$ and $\wh{\ZZ}_N$ mapping $y$ to $\chi_y$ given by $\chi_y(x)=\ze_N^{xy}$, for all $x,y,\in\ZZ_N$, thus considering both $S$ and $T$ as subsets of the same
 group, $\ZZ_N$. Then, putting $\xi=\chi_y$,
 the left hand side of \eqref{formdualdefn} can be written as
 \[\frac{1}{\abs{S}^2}\bigl|\sum_{x\in S}\ze_N^{xy}\bigr|^2=\frac{1}{\abs{S}^2}\abs{S(\ze_N^{y})}^2,\]
 so we can rewrite \eqref{formdualdefn} in the cyclic case as
 \begin{equation}\label{formdualalt}
  \frac{1}{\abs{S}^2}\abs{S(\ze_N^y)}^2=\frac{1}{\abs{T}}\wtn{T}(y).
 \end{equation}
 Furthermore, there is an obvious connection between the mask polynomial of $S$ and the Fourier transform of $\bs1_S$, namely
 \[\hat{\bs1}_S(y)=S(\ze_N^y).\]
 The operator $\frac{1}{\sqrt{N}}\bs F$ is unitary, so Parseval's identity is
 \begin{equation}\label{Parseval}
 N\sum_{x=1}^N\abs{f(x)}^2 = \sum_{y=1}^N\abs{\hat{f}(y)}^2,
 \end{equation}
 and we have $\wh{f\ast g}=\hat{f}\hat{g}$, therefore,
 \begin{equation}\label{equiv1}
 \abs{S(\ze_N^y)}^2=\abs{\hat{\bs1}_S(y)}^2=\hat{\bs1}_S(y)\hat{\bs1}_{-S}(y)=\wh{\wtn{S}}(y),
 \end{equation}
 so \eqref{formdualdefn} can also be written as
 \begin{equation}\label{equiv2}
  \frac{1}{\abs{S}^2}\abs{\hat{\bs1}_S(y)}^2=\frac{1}{\abs{T}}\wtn{T}(y),
 \end{equation}
 when $G=\ZZ_N$.
 Summing over $y\in\ZZ_N$ and applying \eqref{Parseval}, we obtain
 \[\frac{N}{\abs{S}^2}\sum_{y\in\ZZ_N}\abs{\bs1_S(y)}^2=\frac{N}{\abs{S}}=\abs{T},\]
 proving a fact already known \cite{CKRS,Schueler16}.
 \begin{prop}\label{STN}
 Let $T,S\ssq\ZZ_N$ be formally dual. Then $N=\abs{S}\cdot\abs{T}$. 
 \end{prop}
 \eqref{equiv2} can also be written as
 \begin{equation}\label{equiv3}
  \frac{\wh{\wtn{S}}(y)}{\abs{S}^2}=\frac{\wtn{T}(y)}{\abs{T}}.
 \end{equation}
 $\frac{1}{N}\bs F^2$ fixes all even functions, therefore by Fourier inversion we get
 \[\frac{N\wtn{S}(y)}{\abs{S}^2}=\frac{\wh{\wtn{T}}(y)}{\abs{T}}\Longleftrightarrow\frac{\wtn{S}(y)}{\abs{S}}=\frac{\wh{\wtn{T}}(y)}{\abs{T}^2},\]
 confirming that the definition of formal duality is indeed \emph{dual} with respect to $S$, $T$, as expected.
 \begin{prop}\label{basic}
  Let $T,S\ssq\ZZ_N$ be formally dual subsets. Then, the values of $\wtn{T}$ are fixed within a divisor class, and $\abs{T(\ze_N^d)}^2\in\ZZ$ for every integer $d$.
  Also, the mask polynomial of $T-T$ as a multiset is $\equiv T(X)T(X^{-1})\bmod(X^N-1)$, where it is understood that $X^{-1}\equiv X^{N-1}\bmod(X^N-1)$, hence
  \begin{eqnarray*}
  T(X)T(X^{-1}) &\equiv& \sum_{d\mid N}\wtn{T}(d)R_{N/d}(X^d)\bmod(X^N-1)\\
  &\equiv& \frac{\abs{T}}{\abs{S}^2}\sum_{d\mid N}\abs{S(\ze_N^d)}^2R_{N/d}(X^d)\bmod(X^N-1).
  \end{eqnarray*}
 \end{prop}
 
 \begin{proof}
  Let $y\in\ZZ_N$ be arbitrary, and let $g\in\ZZ_N^{\star}$. Consider $\sig\in\Gal(\QQ(\ze_N)/\QQ)$ such that $\sig(\ze_N)=\ze_N^g$. The right hand side of \eqref{formdualalt} is
  a rational number, hence invariant under the action of $\sig$; the left hand side though, becomes 
  \[\frac{1}{\abs{S}^2}\abs{S(\ze_N^{gy})}^2=\frac{1}{\abs{T}}\wtn{T}(gy),\]
  which gives
  \[\wtn{T}(y)=\wtn{T}(gy),\]
  for all $y\in\ZZ_N$, $g\in\ZZ_N^{\star}$, proving the first part. Next, by \eqref{formdualalt} we obtain
  \[\abs{T(\ze_N^d)}^2=\frac{\abs{T}^2}{\abs{S}}\wtn{S}(d),\]
  for every $d$. The left hand side is in $\ZZ[\ze_N]$, while the right hand side in $\QQ$, whence $\abs{T(\ze_N^d)}^2\in\ZZ$ for all $d$ (similarly for $S$).
  
  The next part follows from the fact that $T(X^{-1})$ is the mask polynomial of $-T$, and if $A(X)$, $B(X)$ are the mask polynomials of the multisets $A$, $B$, then 
  $A(X)B(X)$ is the mask polynomial of the sumset $A+B$ (\emph{counting multiplicities}). Furthermore, the coefficients of the mask polynomial of the multiset $T-T$ are
  precisely the values of $\wtn{T}$, which yields
  \begin{eqnarray*}
  T(X)T(X^{-1}) & \equiv &\sum_{t=0}^{N-1}\wtn{T}(t)X^t\equiv\sum_{d\mid N}\wtn{T}(d)R_{N/d}(X^d)\equiv \\
  & \equiv &\frac{\abs{T}}{\abs{S}^2}\sum_{d\mid N}\abs{S(\ze_N^d)}^2R_{N/d}(X^d)\bmod(X^N-1),
  \end{eqnarray*}
  completing the proof.
 \end{proof}

 For any $T\ssq\ZZ_N$ and $d\mid N$, $0\leq j\leq d-1$, we define
 \[T_{j,d}:=\set{t\in T:t\equiv j\bmod d}.\]
 
 \begin{prop}\label{modd}
  The following holds for a primitive $T\ssq\ZZ_N$ possessing a formal dual:
  \[\wtn{T}(1)R_N(X)\equiv \sum_{d\mid N}\mu(d)\sum_{j=0}^{d-1}T_{j,d}(X)T_{j,d}(X^{-1})\bmod(X^N-1).\]
 \end{prop}
 
 \begin{proof}
  Consider the following union of \emph{multisets}:
  \[\bigcup_{\substack{0\leq i,j\leq\rad(N)-1\\ \gcd(N,i-j)=1}}(T_{i,\rad(N)}-T_{j,\rad(N)}).\]
  By definition, this is precisely the set $\ZZ_N^{\star}$, where every element appears with the same multiplicity, i.e. $\wtn{T}(1)$, hence its mask polynomial is
  \[\wtn{T}(1)R_N(X).\]
  On the other hand, we have
  \[\sum_{d\mid N}\mu(d)\sum_{j=0}^{d-1}T_{j,d}(X)T_{j,d}(X^{-1})\equiv \sum_{n=0}^{N-1}\sum_{d\mid N}\mu(d)\sum_{j=0}^{d-1}\wtn{T_{j,d}}(n)X^n\bmod(X^N-1),\]
  so we will compare coefficients between the latter polynomial and $\wtn{T}(1)R_N(X)$. The coefficient of $X^n$
  in $\sum_{n=0}^{N-1}\sum_{d\mid N}\mu(d)\sum_{j=0}^{d-1}\wtn{T_{j,d}}(n)X^n$ is simply
  \begin{equation}\label{coeffs}
  \sum_{d\mid N}\mu(d)\sum_{j=0}^{d-1}\wtn{T_{j,d}}(n).
  \end{equation}
  The term $\wtn{T_{j,d}}(n)$ counts the number of pairs $(t,t')\in T_{j,d}\times T_{j,d}$ that satisfy $t-t'=n$. If $\gcd(n,N)=1$, $t$ and $t'$ cannot belong to the same set
  $T_{j,d}$, for every $d>1$, $0\leq j\leq d-1$; therefore the only contribution comes from the term $\wtn{T}(n)=\wtn{T}(1)$, which is the same as the coefficient of $X^n$ 
  in $\wtn{T}(1)R_N(X)$. If $\gcd(n,N)>1$, then the contribution of a specific pair $(t,t')\in T\times T$ with $t-t'=n$ in \eqref{coeffs} is $\sum_{d\mid t-t'}\mu(d)=0$, which
  shows that both coefficients must be equal to $0$ in this case, completing the proof.
 \end{proof}

 \begin{lemma}\label{mainformula}
  Let $T$, $S$ be formally dual subsets of $\ZZ_N$. Then, for every $d\mid N$ we have
  \[\frac{1}{\sqrt{d}\abs{T}^{3/2}}\sum_{e\mid d}\mu(\tfrac{d}{e})\abs{T(\ze_N^e)}^2=\frac{1}{\sqrt{N/d}\abs{S}^{3/2}}\sum_{\de\mid \frac{N}{d}}\mu(\tfrac{N/d}{\de})\abs{S(\ze_N^{\de})}^2.\]
 \end{lemma}
 
 \begin{proof}
  By Proposition \ref{basic}, we get
  \[\abs{T(\ze_N^e)}^2=\frac{\abs{T}}{\abs{S}^2}\sum_{\de\mid N}\abs{S(\ze_N^{\de})}^2C_{N/\de}(e)=
  \frac{\abs{T}}{\abs{S}^2}\sum_{\de\mid N}\abs{S(\ze_N^{\de})}^2\sum_{g\mid\gcd(e,N/\de)}g\mu(\frac{N/\de}{g}).\]
  Hence,
  \begin{equation}\label{moebius1}
  \sum_{e\mid d}\mu(\frac{d}{e})\abs{T(\ze_N^e)}^2=\frac{\abs{T}}{\abs{S}^2}\sum_{\de\mid N}\abs{S(\ze_N^{\de})}^2\sum_{\substack{e\mid d\\g\mid\gcd(e,N/\de)}}g\mu(d/e)\mu(\frac{N/\de}{g}).
  \end{equation}
  The inner sum  is equal to
  \[\sum_{g\mid\gcd(d,N/\de)}g\mu(\frac{N/\de}{g})\sum_{g\mid e\mid d}\mu(d/e)=\sum_{g\mid\gcd(d,N/\de)}g\mu(\frac{N/\de}{g})\sum_{e'\mid \frac{d}{g}}\mu(\frac{d/g}{e'}).\]
  The latter sum is nonzero, precisely when $g=d$ and $d\mid \frac{N}{\de}$; in that case it's equal to $d\mu(\frac{N/d}{\de})$. Substituting into \eqref{moebius1} we obtain
  \[\sum_{e\mid d}\mu(\frac{d}{e})\abs{T(\ze_N^e)}^2=\frac{d\abs{T}}{\abs{S}^2}\sum_{\de\mid \frac{N}{\de}}\mu(\frac{N/d}{\de})\abs{S(\ze_N^{\de})}^2,\]
  and the proof is completed by dividing both sides by $\sqrt{d}\abs{T}^{3/2}$.
 \end{proof}

 Next, we restrict our attention to \emph{primitive} subsets of $\ZZ_N$, that is, subsets that are not contained in a coset of a proper subgroup. We remark that if we remove this restriction, there
 are trivial pairs of formal duals given by $H<\ZZ_N$ and $H^{\perp}$, its orthogonal subgroup (which is isomorphic to the group of characters that vanish on $H$). However, these examples come from 
 dual lattices in Euclidean spaces, so we gain no new information with regards to formal duality there. Furthermore, if $T$ is not primitive, we can always reduce this situation to a primitive
 formally dual pair in $\ZZ_M$, where $M\mid N$ \cite{CKRS}.
 
 \begin{prop}\label{upperboundmask}
  Let $T\ssq\ZZ_N$ be primitive. Then, for every $\ze\neq1$, $\ze^N=1$, we have
  \[\abs*{T(\ze)}<\abs{T}.\]
 \end{prop}
 \begin{proof}
  Assume otherwise, that $\abs{T(\ze)}^2=\abs{T}^2$, where $\ze$ is a primitive $d$th root of unity. As $T$ is primitive, there are at least two integers $j,k$ with $j\not\equiv k\bmod d$, such that
  $T_{j,d}$ and $T_{k,d}$ are nonempty. Then,
  \[
  \abs{T(\ze)}=\abs*{\sum_{i=0}^{d-1}T_{i,d}\ze^i}\leq \abs{T_{j,d}\ze^j+T_{k,d}\ze^k}+\sum_{\substack{0\leq i\leq d-1\\ j,k\not\equiv i\bmod d}}\abs{T_{i,d}}<\sum_{i=0}^{d-1}\abs{T_{i,d}}=\abs{T},
  \]
  as $\abs{T_{j,d}\ze^j+T_{k,d}\ze^k}<\abs{T_{j,d}}+\abs{T_{k,d}}$; equality could only hold if $\abs{T_{j,d}}\ze^j=\la\abs{T_{k,d}}\ze^k$ for some $\la>0$, however, this is impossible.
 \end{proof}

 Lemma \ref{mainformula} provides the following estimate on $\wtn{T}(1)$.
 
 \begin{cor}\label{firstestimate}
  Let $T$, $S$ be primitive formally dual subsets of $\ZZ_N$. Then,
  \[\wtn{T}(1)\leq\frac{2^{\om(N)-1}\abs{T}^2}{N},\]
  and similarly for $S$. Equality can only hold when $\om(N)=1$ %$\abs{T}=\abs{S}=\sqrt{N}$ (i.e. $N$ is an even power of a prime), 
  and $T(\ze_p)=0$. Furthermore, if $\abs{T}\leq\abs{S}$ and $\wtn{T}(1)\neq0$ (or equivalently, $S(\ze_N)\neq0$), the following inequalities hold:
  \[\sqrt{\frac{N}{2^{\om(N)-1}}}\leq\abs{T}\leq\sqrt{N}\leq\abs{S}\leq\sqrt{2^{\om(N)-1}N},\]
  where again the leftmost and rightmost inequalities are equalities precisely when $\om(N)=1$, $T(\ze_p)=0$ and $\wtn{T}(1)=1$.
 \end{cor}
 
 \begin{proof}
  We apply Lemma \ref{mainformula} for $d=N$:
  \[\frac{1}{\abs{S}^{3/2}}\abs{S(\ze_N)}^2=\frac{1}{\sqrt{N}\abs{T}^{3/2}}\sum_{e\mid N}\mu(N/e)\abs{T(\ze_N^e)}^2\leq\frac{2^{\om(N)-1}\abs{T}^2}{\sqrt{N}\abs{T}^{3/2}},\]
  by Proposition \ref{upperboundmask},
  where equality could only hold if $\om(N)=1$ and $T(\ze_p)=0$.
  By \eqref{formdualalt}, the left hand side is equal to
  \[\frac{\abs{S}^{3/2}}{N}\wtn{T}(1),\]
  whence
  \[\wtn{T}(1)\leq\frac{2^{\om(N)-1}N\abs{T}^2}{\sqrt{N}\abs{S}^{3/2}\abs{T}^{3/2}}=\frac{2^{\om(N)-1}\abs{T}^2}{N},\]
  as desired. Solving this inequality for $\abs{T}$ and then using $N=\abs{S}\cdot\abs{T}$ by Proposition \ref{STN}, we get the final inequalities.
 \end{proof}
 
 The above estimate is not always the best we can achieve; we will close this section with another estimate which most of the times is better.
 
 \begin{lemma}\label{mainestimate}
  Let $T$, $S$ be primitive formally dual subsets of $\ZZ_N$. Then
  \begin{equation}\label{main1}
   \wtn{T}(1)\leq\frac{\abs{T}^2-\abs{T}}{\vphi(N)}<\frac{\abs{T}^2}{\vphi(N)},
  \end{equation}
  and similarly for $S$. Furthermore, if $\abs{T}\leq\abs{S}$ and $\wtn{T}(1)\neq0$ (or equivalently, $S(\ze_N)\neq0$), the following inequalities hold:
  \begin{equation}\label{main2}
  \sqrt{\vphi(N)}<\abs{T}\leq\sqrt{N}\leq\abs{S}<\frac{N}{\sqrt{\vphi(N)}}. 
  \end{equation}
  The middle equalities can only hold when $N$ is a square.
 \end{lemma}
 
 \begin{proof}
  The number of nonzero differences between elements of $T$ are precisely $\abs{T}^2-\abs{T}$ (counting multiplicities). If $\wtn{T}(1)\neq 0$, this means that every element of $\ZZ_N^{\star}$
  appears at least once in $T-T$, yielding \eqref{main1} (if $\wtn{T}(1)=0$, it is trivial). From $\abs{T}\leq\abs{S}$ and \eqref{main1} we easily obtain \eqref{main2}.
 \end{proof}
 
 Now we compare the inequalities from Corollary \ref{firstestimate} and Lemma \ref{mainestimate}. Suppose that $\om(N)\geq3$, that is $N$ has at least three distinct prime factors.
 Then,
 \[\frac{\abs{T}^2}{\vphi(N)}=\frac{\abs{T}^2}{N}\prod_{\substack{p\mid N\\ p\text{ prime }}}\bra{1-\frac{1}{p}}^{-1}\leq
 \frac{\abs{T}^2}{N}\cdot 2\cdot\frac{3}{2}\cdot\bra{\frac{5}{4}}^{\om(N)-2}<\frac{2^{\om(N)-1}\abs{T}^2}{N},\]
 as $(8/5)^k>3/2$ for all $k\geq1$, so Lemma \ref{mainestimate} provides a better bound. The same holds when $\om(N)=2$ and $N$ is odd, as
 \[\frac{\abs{T}^2}{\vphi(N)}\leq\frac{\abs{T}^2}{N}\frac{3\cdot 5}{2\cdot 4}<\frac{2\abs{T}^2}{N},\]
 but at all other cases, Corollary \ref{firstestimate} gives better bounds. Indeed, if $\om(N)=1$, then obviously
 \[\frac{\abs{T}^2}{N}<\frac{\abs{T}^2}{\vphi(N)},\]
 while if $\om(N)=2$ and $N$ even, we obtain
 \[\frac{2\abs{T}^2}{N}=\frac{2\abs{T}^2}{\vphi(N)}\frac{p-1}{2p}<\frac{\abs{T}^2}{\vphi(N)},\]
 where $p$ the unique odd prime dividing $N$.

 \bigskip
 \section{The field descent method}\label{fielddescent}
 \bigskip
 
  We mention the main tools from the \emph{field descent method}, developed in \cite{KaSchmidt,Schmidt99,SchmidtFDM}. The question
  that was addressed by this method is the following: under which circumstances can we have $X\in\ZZ[\ze_N]$, such that $\abs{X}^2=n\in\ZZ$? 
  First, we need the definition below, before we pass to the main theorems of the field descent method.% (see also Definition 2.6 in \cite{KaSchmidt}).
  
  \begin{defn}[Definition 2.6 \cite{KaSchmidt}]\label{Fmn}
   Let $m,n>1$ integers. $\mathscr{D}(t)$ denotes the set of prime divisors of an integer $t$. For $q\in\mathscr{D}(n)$ let
   \[m_q:=\begin{cases}
           \prod_{p\in\mathscr{D}(m)\sm\set{q}}p,\	\	&\text{if }m\text{ is odd or }q=2\\
           4\prod_{p\in\mathscr{D}(m)\sm\set{2,q}}p,\	\	&\text{otherwise.}
          \end{cases}
\]
Set
\begin{eqnarray*}
b(2,m,n)=&\max_{q\in\mathscr{D}(n)\sm\set{2}}\set{\nu_2(q^2-1)+\nu_2(\ord_{m_q}(q))-1}\\
b(r,m,n)=&\max_{q\in\mathscr{D}(n)\sm\set{r}}\set{\nu_r(q^{r-1}-1)+\nu_r(\ord_{m_q}(q))}
\end{eqnarray*}
for any prime $r>2$ with the convention that $b(2,m,n)=2$ if $\mathscr{D}(n)=\set{2}$ and $b(r,m,n)=1$ if $\mathscr{D}(n)=\set{r}$. We define
\[F(m,n):=\gcd(m,\prod_{p\in\mathscr{D}(m)}p^{b(p,m,n)}).\]
  \end{defn}

  \begin{thm}[\cite{Schmidt99,SchmidtFDM}]\label{FDM}
   Let $A\in\ZZ[\ze_m]$, such that $\abs{A}^2=n$. Then, $A$ belongs to a smaller cyclotomic field up to multiplication by a root of unity, that is
   \[A\in \ze_m^j\ZZ[\ze_{F(m,n)}].\]
  \end{thm}

  \begin{thm}[\cite{SchmidtFDM}]\label{Fbound}
   Let $X\in\ZZ[\ze_m]$ be of the form 
   \[X=\sum_{i=0}^{m-1}a_i\ze_m^i\]
   with $0\leq a_i\leq C$ for some constant $C$ and assume that $n=\abs{X}^2$ is an integer. Then
   \[n\leq \frac{C^2F(m,n)^2}{4\vphi(F(m,n))}.\]
  \end{thm}
  
  The definition of $F(m,n)$ seems technical, so we need the Proposition below in order to shed some light on it; see also \cite{KaSchmidt}.
  
  \begin{prop}\label{Fproperties}
   The number $F(m,n)$ has the following properties.
   \begin{enumerate}
    \item $F(m,n)$ divides $m$.
    \item $\rad(m)=\rad(F(m,n))$.
    \item $F(m,n)=F(m,\rad(n))$, i.e., if we fix $m$, $F(m,n)$ depends only on the prime divisors of $n$.
    \item For every finite set of primes $P$, there is an explicitly computable constant $C(P)$, such that $F(m,n)\leq C(P)$ whenever $\mathscr{D}(m),\mathscr{D}(n)\ssq P$
    (Proposition 2.2.7 \cite{SchmidtFDM}).
   \end{enumerate}
  \end{prop}
  
  The case $\om(m)\leq 2$ and $\mathscr{D}(n)\ssq\mathscr{D}(m)$ ($n>1$) will be particularly useful in the next two sections, so we will provide formulae for $F(m,n)$ in this case.
  %When $T$, $S$ are primitive formally dual subsets of $\ZZ_N$ with $\abs{T}\leq\abs{S}$, where $\rad(N)=pq$ with $p$, $q$ distinct primes, then according to Proposition \ref{Som}
  %$\rad(\abs{S(\ze_N)}^2)\mid pq$. Since we wish to apply Theorem \ref{Fbound} in this case, we will first compute $F(m,n)$ when both $m$ and $n$ are divisible by two primes at most,
  %say $p$ and $q$. By (3) of the above Remark, $F(m,n)$ must then be equal to $F(m,pq)$, $F(m,p)$, or $F(m,q)$.
  By Proposition \ref{Fproperties}(3), $F(m,n)=F(m,p)$ if $\mathscr{D}(m)=\set{p}$, and $F(m,n)$ is equal to $F(m,pq)$, $F(m,p)$, or $F(m,q)$, if $\mathscr{D}(m)=\set{p,q}$.
  
  \begin{prop}\label{Cpq}
  Let $p$, $q$ be distinct primes. We have
  \[F(p^k,p)=\begin{cases}
              p,\	\	&\text{ if }p>2\\
              \gcd(2^k,4),\	\	&\text{ if }p=2.
             \end{cases}
\]
   Next, let $m=p^kq^l$ and $k,l>0$. If $m$ is odd we have
   \begin{eqnarray}
   F(m,pq) &=& \gcd(m,(p^{q-1}-1)(q^{p-1}-1)) \label{Fmpq}\\
   F(m,p) &=& p\gcd(q^l,p^{q-1}-1) \label{Fmp}\\
   F(m,q) &=& q\gcd(p^k,q^{p-1}-1) \label{Fmq},
   \end{eqnarray}
   and if $m$ is even (without loss of generality, $p=2$) we have
   \begin{eqnarray}
   F(m,2q) &=& \begin{cases}
		\gcd(m,\frac{1}{2}(q^2-1)(2^{q-1}-1)),\		\	&\text{ if }q\equiv1\bmod4\\ 
                \gcd(m,(q^2-1)(2^{q-1}-1)),\		\	&\text{ if }q\equiv3\bmod4\\
               \end{cases}
\label{Fm2q}\\
   F(m,2) &=& \begin{cases}
	      2\gcd(q^l,2^{q-1}-1),\		\	&\text{ if }4\nmid m\\
              4\gcd(q^l,2^{q-1}-1),\		\	&\text{ if }4\mid m\\
              \end{cases}
 \label{Fm2}\\
   F(m,q) &=& \begin{cases}
	      q\gcd(2^k,\frac{1}{2}(q^2-1)),\		\	&\text{ if }q\equiv1\bmod4\\
	      q\gcd(2^k,q^2-1),\		\	&\text{ if }q\equiv3\bmod4\\               
              \end{cases}
 \label{Fmqeven}.
   \end{eqnarray}
   At all cases, 
   \begin{equation}\label{pqconst}
    F(m,p),F(m,q)\leq F(m,pq) \leq p^aq^b,
   \end{equation}
   where 
   \[a=\begin{cases}
	\nu_p(q^{p-1}-1), \	\	&\text{ if }p>2\\
	\nu_2(q^2-1), \	        \	&\text{ otherwise},
       \end{cases}
\]
  $b=\nu_q(p^{q-1}-1)$.
  \end{prop}
  
  \begin{proof}
  By Definition \ref{Fmn}, $b(2,2^k,2)=2$ and $b(p,p^k,p)=1$ for $p>2$, so $F(2^k,2)=\gcd(2^k,2^2)$, while $F(p^k,p)=\gcd(p^k,p)=p$, proving the first part.
  
   Next, let $m=p^kq^l$. First we assume that $m$ is odd. Then, according to Definition \ref{Fmn}, $m_q=p$ and $m_p=q$, therefore $b(p,m,pq) = b(p,m,q) = \nu_p(q^{p-1}-1)$ while $b(p,m,p)=1$, and similarly
   $b(q,m,pq)=b(q,m,p)=\nu_q(p^{q-1}-1)$ and $b(q,m,q)=1$, since $p\nmid \ord_p(q)$ and $q\nmid \ord_q(p)$. This shows that
   \[F(m,pq)=\gcd(m,p^{b(p,m,pq)}q^{b(q,m,pq)})=\gcd(m,(p^{q-1}-1)(q^{p-1}-1)),\]
   since
   \[\nu_p((p^{q-1}-1)(q^{p-1}-1))=\nu_p(p^{q-1}-1)+\nu_p(q^{p-1}-1)=b(p,m,pq),\]
   and similarly for $b(q,m,pq)$. Moreover,
   \[F(m,p)=\gcd(m,pq^{b(q,m,p)})=p\gcd(q^l,q^{p-1}-1),\]
   and the formula for $F(m,q)$ is recovered in the same manner.
   
   Next, assume that $m$ is even, so that $p=2$ without loss of generality. Then, $m_2=q$ and $m_q=4$, therefore $b(2,m,2q)=b(2,m,q)=\nu_2(q^2-1)+\nu_2(\ord_4(q))-1$ and $b(2,m,2)=2$;
   similarly $b(q,m,2q)=b(q,m,2)=\nu_q(2^{q-1}-1)$, as $q\nmid\ord_q(2)$, and $b(q,m,q)=1$, as before. Hence,
   \[F(m,2q)=\gcd(m,2^{b(2,m,2q)}q^{b(q,m,2q)})=\gcd(m,\frac{1}{2}\ord_4(q)(q^2-1)(2^{q-1}-1))\]
   yielding \eqref{Fm2q}. Also
   \[F(m,2)=\gcd(m,4q^{b(q,m,2)})=2\gcd(\frac{m}{2},2^q-2),\]
   and
   \[F(m,q)=\gcd(m,2^{b(2,m,q)}q)=q\gcd(2^k,\frac{1}{2}\ord_4(q)(q^2-1)),\]
   giving us \eqref{Fm2} and \eqref{Fmqeven}. Equation \eqref{pqconst} follows easily.
  \end{proof}

 \bigskip
 \section{The prime power case, revisited}\label{oneprime}
 \bigskip
 
 Now we are ready to apply the methods already introduced; we will begin by providing two different proofs for the prime power case, one using the field descent method and one using the polynomial method.
 We emphasize that the polynomial method comprises of similar arguments as in \cite{Schueler16}, albeit with a different language.
 
 \begin{thm}\label{primepower}
  Let $N=p^k$, where $p$ prime. Then, $\ZZ_N$ cannot have primitive formally dual sets, unless $N=4$
 \end{thm}

 Let $N=p^m$, and let $T$, $S$ be a pair of primitive formally dual subsets of $\ZZ_N$. We have (see also Lemma \ref{intersection})
 \[(T-T)\cap\ZZ_N^{\star}\neq\vn\neq(S-S)\cap\ZZ_N^{\star},\]
 otherwise $T-T\ssq p\ZZ_N$ or $S-S\ssq p\ZZ_N$, a possibility excluded due to the fact that $T$, $S$ are primitive. 
 This is equivalent to $\wtn{T}(1)\wtn{S}(1)\neq0$ or $T(\ze_N)S(\ze_N)\neq0$.
 Without loss of generality, we suppose $\abs{T}\leq \abs{S}$,
 so that $\abs{T}\leq\sqrt{N}$. If $N$ is not a square, then $m=2k+1$, and
 \[\abs{T}\leq p^k<p^{k+1}\leq\abs{S}.\]
 But then, the differences between unequal elements of $T$ are
 \[\abs{T}^2-\abs{T} \leq p^{2k}-p^k<p^{2k}(p-1)=\vphi(N),\]
 contradicting Lemma \ref{mainestimate}. Hence, $N$ must be a square, so that $m=2k$. If $\abs{T}<\abs{S}$, then
 $\abs{T}\leq p^{k-1}$, and we are led again to a contradiction, as
 \[\abs{T}(\abs{T}-1)\leq p^{2k-2}-p^{k-1}<p^{2k-1}(p-1)=\vphi(N).\]
 Therefore, $\abs{T}=\abs{S}=\sqrt{N}=p^k$.
 
 \begin{proof}[Proof of Theorem \ref{primepower} via the field descent method.]
  By Lemma \ref{mainestimate} we have 
  \[0<\wtn{T}(1)<\frac{\abs{T}^2}{N}=\frac{N}{\vphi(N)}=\frac{p}{p-1}\leq2,\]
  hence $\wtn{T}(1)=1$ and by \eqref{formdualalt} we get
  \[n=\abs{S(\ze_N)}^2=\frac{\abs{S}^2}{\abs{T}}\geq\sqrt{N}=p^k,\]
  and $\mathscr{D}(n)=\set{p}$. Furthermore, $S(\ze_N)=\sum_{s\in S}\ze_N^s$, therefore applying Theorem \ref{Fbound} with $C=1$ and Propositions \ref{Fproperties} and \ref{Cpq} we get for $p$ odd
  \[p^k\leq\frac{F(p^{2k},p)^2}{4\vphi(F(p^k,p))}=\frac{p^2}{4(p-1)}<p,\]
  a contradiction, while for $p=2$,
  \[2^k\leq\frac{F(2^{2k},2)^2}{4\vphi(F(2^{2k},2))}=2,\]
  yielding $k=1$, as desired.
 \end{proof}
 
 \medskip

 \begin{proof}[Proof of Theorem \ref{primepower} via the polynomial method.]
 Define $T_j:=\set{t\in T: t\equiv j\bmod p}$. Since 
 \[\wtn{T}(1)=1,\]
 Proposition \ref{modd} gives
 \[R_N(X)\equiv T(X)T(X^{-1})-\sum_{j=0}^{p-1}T_j(X)T_j(X^{-1})\bmod X^N-1,\]
 so that
 \[R_N(\zeta)=\abs{T(\zeta)}^2-\sum_{j=0}^{p-1}\abs{T_j(\zeta)}^2,\]
 for every $N$th root of unity $\zeta$. Putting $\zeta=1$ we get
 \begin{eqnarray*}
 \vphi(N)=N-\sum_{j=0}^{p-1}\abs{T_j}^2 \Longleftrightarrow \sum_{j=0}^{p-1}\abs{T_j}^2 = N/p \Longleftrightarrow N= p\sum_{j=0}^{p-1}\abs{T_j}^2\geq \abs{T}^2
 \end{eqnarray*}
 by Cauchy-Schwarz inequality, where equality holds precisely when all $\abs{T_j}$ are the same; equality indeed holds\footnote{In \cite{CKRS,Schueler16}, this
 argument was attributed to Gregory Minton.}, since $\abs{T}=\sqrt{N}$, hence $\abs{T_j}=p^{k-1}$, for
 all $j$. This implies that
 \[T(\ze_p)=\sum_{j=0}^{p-1}\ze_p^j\abs{T_j}=0,\]
 and similarly $S(\ze_p)=0$. The case $N=p^2$ has already been tackled in \cite{CKRS}, where primitive formally dual subsets exist only for $p=2$, so we may assume that
 $m=2k$ with $k\geq2$.
 
 If $T(\zeta)\neq0$ for any other $N$th root of unity besides $\ze_p$ and its conjugates, then $S-S$ intersects all divisor classes except for $\frac{N}{p}\ZZ_N^{\star}$,
 yielding
 \[\abs{S}^2-\abs{S}\geq \sum_{i=0}^{m-2}\vphi(N/p^i)=N-p,\]
 while on the other hand
 \[\abs{S}^2-\abs{S}=N-\sqrt{N}<N-p,\]
 a contradiction. So, there is some $N$th root of unity $\zeta$ with $\zeta^p\neq1$, such that $T(\zeta)=0$, hence by \eqref{calcrsums}
 \[0=R_N(\zeta)=-\sum_{j=0}^{p-1}\abs{T_j(\zeta)}^2,\]
 implying that $T_j(\zeta)=0$ for all $j$. This leads to a contradiction, when we take derivatives; by Lemma \ref{Rderivatives} we have
 %\[R'(X)=\ent{\frac{X^N-1}{X-1}-\frac{X^N-1}{X^p-1}}'=(X^N-1)\ent{-\frac{1}{(X-1)^2}}+\frac{pX^{p-1}}{(X^p-1)^2}+NX^{N-1}\ent{\frac{1}{X-1}-\frac{1}{X^p-1}},\]
 %therefore
 \[R'(\zeta)\equiv N\zeta^{-1}\ent{\frac{1}{\zeta-1}-\frac{1}{\zeta^p-1}}\equiv\frac{N(1+\zeta+\dotsc+\zeta^{p-2})}{\zeta^p-1}\equiv\frac{N}{\zeta^p-1}\frac{\zeta^{p-1}-1}{\zeta-1}
 \not\equiv 0\bmod N\ZZ[\ze_N],\]
 since $\frac{\zeta^{p-1}-1}{\zeta-1}$ is a unit in $\ZZ[\ze_N]$, while $\zeta^p-1$ is not. On the other hand, we obtain
 \begin{eqnarray*}
 & \ent{T(X)T(X^{-1})-\sum_{j=0}^{p-1}T_j(X)T_j(X^{-1})}' = \\
 & T'(X)T(X^{-1})-T(X)T'(X^{-1})X^{-2}-\sum_{j=0}^{p-1}(T'_j(X)T_j(X^{-1})-T_j(X)T'_j(X^{-1})X^{-2}),
 \end{eqnarray*}
 and keeping in mind that we have $T(\zeta)=T(\zeta^{-1})=T_j(\zeta)=T_j(\zeta^{-1})=0$ for all $j$, we get
 \[\frac{d}{dX}\ent{T(X)T(X^{-1})-\sum_{j=0}^{p-1}T_j(X)T_j(X^{-1})}_{X=\zeta}=0,\]
 a contradiction. Thus, if $N$ is a prime power with $N\neq4$, there are no primitive formally dual pairs in $\ZZ_N$.
 \end{proof}

 \bigskip
 \section{Products of two prime powers}\label{twoprimes}% and the field descent method}
 \bigskip
 
 In this section, we will address the case $\mathscr{D}(N)=\set{p,q}$.
 
 \begin{lemma}\label{intersection}
  Let $T\ssq\ZZ_N$ be primitive and $N$ is divisible by at most two primes. Then $(T-T)\cap\ZZ_N^{\star}\neq\vn$.
 \end{lemma}
 
 \begin{proof}
  If $\rad(N)=p$, then $T$ is not a subset of any coset of the subgroup $p\ZZ_N$. In particular, there are two elements of $T$ that are not congruent $\bmod p$; hence, their difference
  is also prime to $N$, which yields the desired conclusion.
  
  Let $\rad(N)=pq$, and define $T_j=\set{t\in T : t\equiv j\bmod p}$. Since $T$ is primitive, at least two of the sets $T_j$ are nonempty. Let $T_a$ and $T_b$ be two arbitrary nonempty
  sets of this family. Suppose that $T_a$ has at least two elements, say $t$, $t'$, and take $t''\in T_b$. Assume that $(T-T)\cap\ZZ_N^{\star}=\vn$. Since $p\nmid t-t''$, we must certainly
  have $q\mid t-t''$, otherwise $\gcd(t-t'',N)=1$. The same holds for the difference between elements of $T_a$ and $T_b$, so
  \[T_a-T_b\ssq q\ZZ_N.\]
  On the other hand, $t-t'=(t-t'')-(t'-t'')\in q\ZZ_N$, and since $t,t'\in T_a$ were arbitrary, we obtain
  \[T_a-T_a\ssq q\ZZ_N,\]
  therefore,
  \[T-T=\bigcup_{a,b=0}^{p-1}(T_a-T_b)\ssq q\ZZ_N,\]
  a contradiction. Thus, $(T-T)\cap\ZZ_N^{\star}\neq\vn$.
  \end{proof}
  
  \begin{rem}
   The above does not necessarily hold when $n$ is divisible by at least three primes. For example, take $\rad(N)=pqr$, and
   \[T=pq\ZZ_N\cup qr\ZZ_N \cup pr\ZZ_N.\]
   $T$ is indeed a primitive set, but $(T-T)\cap\ZZ_N^{\star}=\vn$.
  \end{rem}

  Lemma \ref{intersection} shows that $S(\ze_N)T(\ze_N)\neq0$ when $\rad(N)=pq$ and $S$, $T$ are primitive formally dual, or equivalently, $\wtn{T}(1),\wtn{S}(1)\geq1$. Upper bounds
  are given by the following:
  
  \begin{prop}\label{Som}
   Let $N$ be a positive integer with $\rad(N)=pq$, and $T$, $S$ primitive formally dual subsets of $\ZZ_N$ with $\abs{T}\leq\abs{S}$. Then
   \begin{equation}\label{Someq}
   \abs{S(\ze_N)}^2=\frac{\abs{S}^3}{N}.
   \end{equation}
   Furthermore, $\wtn{S}(1)\leq3$, %if $N$ is even with $3\nmid N$ then $\wtn{S}(1)\leq 4$, and if $6\mid N$ then $\wtn{S}(1)\leq5$.
   so that
   \[\abs{T(\ze_N)}^2=\frac{K\abs{T}^3}{N},\]
   where $1\leq K\leq 3$.
  \end{prop}
  
  \begin{proof}
   By Corollary \ref{firstestimate} we obtain
   \[\wtn{T}(1)<\frac{2\abs{T}^2}{N}\leq2,\]
   whence $\wtn{T}(1)=1$, due to Lemma \ref{intersection}, yielding \eqref{Someq}. Furthermore,
   \[\wtn{S}(1)<\frac{2\abs{S}^2}{N}<\frac{2\cdot2N}{N}=4,\]
   concluding the proof.
   %If $N$ is odd, then by Lemma \ref{mainestimate} we get
   %\[\wtn{S}(1)<\frac{\abs{S}^2}{\vphi(N)}<\frac{N^2}{\vphi(N)^2}=\bra{\frac{pq}{(p-1)(q-1)}}^2\leq\bra{\frac{3\cdot 5}{2\cdot4}}<4.\]
   %If $N$ is even (say $p=2$), then Corollary \ref{firstestimate} gives
   %\[\wtn{S}(1)<\frac{2\abs{S}^2}{N}<\frac{2N}{\vphi(N)}=\frac{4q}{q-1}.\]
   %So, if $q>3$ we obtain $\wtn{S}(1)<5$, while if $q=3$ we get $\wtn{S}(1)<6$, completing the proof.
  \end{proof}

  Since $\abs{S}\mid N$ and $\abs{S(\ze_N)}^2\in\ZZ$, Proposition \ref{Som} implies that $\abs{S(\ze_N)}^2$ is not divisible by primes other than $p$ and $q$, and $\abs{S}\geq\sqrt{N}$
  yields
  \[\abs{S(\ze_N)}^2\geq\sqrt{N},\]
  that is, the squared modulus of the algebraic integer $S(\ze_N)\in\ZZ[\ze_N]$ is a relatively large integer with respect to the order of $\ze_N$, i.e. $N$.   
  This leads to the main result of this section.
  
  \begin{thm}\label{finexc}
   Fix $p$, $q$, two distinct primes. Then, possibly with finitely many exceptions, no group $\ZZ_N$ with $\rad(N)=pq$ possesses primitive formally dual subsets.
  \end{thm}
  
  \begin{proof}
   Let $P=\set{p,q}$ and put $a$, $b$, exactly as in Proposition \ref{Cpq}, hence $F(m,n)\leq p^aq^b$, whenever $\mathscr{D}(n)\ssq\mathscr{D}(m)=P$. Now let $N=p^kq^l$, and
   $T,S\ssq\ZZ_N$ be primitive formally dual, with $\abs{T}\leq\abs{S}$, so that $\abs{S}\geq\sqrt{N}$; also, put $n=\abs{S(\ze_N)}^2$, where $\ze_N=e^{2\pi i/N}$. 
   Then, by Proposition \ref{Som} we get
   \begin{equation}\label{lower}
   n=\frac{\abs{S}^3}{N}\geq\sqrt{N}.
   \end{equation}
   On the other hand, we observe that $S(\ze_N)=\sum_{j\in S}\ze_N^j$, so if we put $X=S(\ze_N)$ as in Theorem \ref{Fbound}, the constant $C$ can be taken equal to $1$. Applying this Theorem
   we obtain
   \begin{equation}\label{upper}
   n\leq \frac{F(N,n)^2}{4\vphi(F(N,n))}=\frac{F(N,n)}{4(1-\frac{1}{p})(1-\frac{1}{q})}\leq\frac{p^aq^b}{4(1-\frac{1}{p})(1-\frac{1}{q})}.
   \end{equation}
   Combining equations \eqref{lower} and \eqref{upper}, we get
   \[N\leq \frac{p^{2a}q^{2b}}{16(1-\frac{1}{p})^2(1-\frac{1}{q})^2},\]
   hence at most finitely many groups $\ZZ_N$ with $\rad(N)=pq$ possess primitive formally dual subsets.
  \end{proof}

Theorem \ref{finexc} tackles the Conjecture when the sum of the exponents of $p$ and $q$ is sufficiently high. Next, we will tackle some cases where one or both of the exponents are
small.

\begin{prop}\label{singleq}
 Let $N=p^aq$, where $p$, $q$ distinct primes. Then $\ZZ_N$ does not have primitive formally dual subsets.
\end{prop}

\begin{proof}
  We remind that if $a=1$ then $N$ is square--free, so this is already proven by \cite{XPC}; we may assume that $a>1$.
 Suppose on the contrary that $T$, $S$ are such subsets, with $\abs{T}<\abs{S}$ as usual (equality cannot occur in this case, as $N$ is not a square), therefore
 \begin{equation}\label{sqrt}
 \sqrt{\vphi(N)}<\abs{T}<\sqrt{N}<\abs{S}<\frac{N}{\sqrt{\vphi(N)}}.
 \end{equation}
 
 By Proposition \ref{Som} we get
 \[\abs{S(\ze_N)}^2=\frac{\abs{S}^3}{N},\]
 hence $q\mid\abs{S}$ and $q\nmid\abs{T}$, so $\abs{T}=p^b$ and $\abs{S}=p^{a-b}q$ for some integer $b$. We recall that
 \begin{equation}\label{Tom}
 \abs{T(\ze_N)}^2=\frac{\abs{T}^3}{N}\bs1_S\ast\bs1_{-S}(1).
 \end{equation}
 Since the latter is an integer, $q\mid N$ and $q\nmid\abs{T}$, we must have 
 \begin{equation}\label{qS}
 q\mid\wtn{S}(1).
 \end{equation} 
 Suppose first that $q$ is odd; then by Proposition \ref{Som} we must have $q=\wtn{S}(1)=3$
 so that $N=3p^a$ and $\vphi(N)=2(p-1)p^{a-1}$. By virtue of Lemma \ref{mainestimate} we obtain
 \[\sqrt{2(p-1)}p^{\frac{a-1}{2}}<p^b<\sqrt{3}p^{\frac{a}{2}}.\]
 If $N$ is odd, then $p>3$, so we have $\sqrt{2(p-1)}>\sqrt{p}$, so the above inequalities lead to
 \[p^{\frac{a}{2}}<p^b<p^{\frac{a+1}{2}}\]
 which is a contradiction, as there is no such integer $b$.
 
 Now suppose that $p=2$, so that $N=2^a3$. Then $\abs{T}=2^b$ and by \eqref{sqrt} we get
 \[2^{\frac{a}{2}}<2^b<2^{\frac{a}{2}}\sqrt{3},\]
 yielding $\abs{T}=2^{\frac{a+1}{2}}$ and $\abs{S}=2^{\frac{a-1}{2}}3$; also, $a$ must be odd. Equations \eqref{Tom} and \eqref{qS} yield
 \[\abs{T(\ze_N)}^2=2^{\frac{a+3}{2}}.\]
 Applying Theorem \ref{Fbound} and equation \eqref{Fm2q}, we get
 \[\abs{T(\ze_N)}^2\leq\frac{F(2^a3,2)}{4\vphi(F(2^a3,2))}=\frac{12^2}{4\vphi(12)}=9,\]
 and since $a$ is odd, the only solution we get from $2^{\frac{a+3}{2}}\leq9$ is $a=3$. Hence, $N=24$, $\abs{T}=4$, $\abs{S}=6$, and $\abs{T(\ze_N)}^2=8$. Now we will try to determine
 $T-T$ as a multi-set. Since $12=\abs{T}^2-\abs{T}<2\vphi(24)=16$ it only contains $\ZZ_{24}^{\star}$ with multiplicity one. The rest of the $4=\abs{T}^2-\abs{T}-\vphi(24)$ differences
 between different elements of $T$ must belong to other divisor classes. Two of these classes must necessarily be from 
 \[12\ZZ_{24}^{\star}=\set{12},\	6\ZZ_{24}^{\star}=\set{6,18},\	3\ZZ_{24}^{\star}=\set{3,9,15,21},\]
 otherwise $4\mid \abs{S}$, a contradiction. Thus, the only possibility is for $12\ZZ_{24}^{\star}$ to appear with multiplicity $2$ and $6\ZZ_{24}^{\star}$ with
 multiplicity $1$. This means, that there is some $t\in T$, such that $t+12\in T$, say $T=\set{u,v,t,t+12}$. But then,
 \[\abs{T(\ze_N)}^2=\abs{\ze_N^u+\ze_N^v+\ze_N^t-\ze_N^t}^2=\abs{\ze_N^u+\ze_N^v}^2\leq 4<8,\]
 a contradiction. Therefore, neither in this case do exist primitive formally dual subsets.
 
 Finally, suppose that $q=2$, so that $N=2p^a$, $\abs{T}=p^b$, $\abs{S}=2p^{a-b}$. By \eqref{sqrt} we have the bounds
 \[\sqrt{2}p^{\frac{a}{2}}<2p^{a-b}<\frac{2p^{\frac{a+1}{2}}}{\sqrt{p-1}},\]
 or equivalently,
 \[\frac{1}{\sqrt{2}}p^{\frac{a}{2}}<p^{a-b}<\frac{p^{\frac{a+1}{2}}}{\sqrt{p-1}},\]
 which can only happen if $a$ is even and $a=2b$. Hence, $\abs{S}=2p^\frac{a}{2}$ and $\abs{T}=p^{\frac{a}{2}}$; however, this contradicts Corollary \ref{firstestimate}
 and Lemma \ref{intersection}, as
 \[\wtn{T}(1)<\frac{2\abs{T}^2}{N}=1,\]
 completing the proof.
\end{proof}

\begin{prop}\label{poddq2}
 Let $N=p^aq^2$, where $p$, $q$ distinct primes and $a$ is odd. Then $\ZZ_N$ does not have primitive formally dual subsets.
\end{prop}

\begin{proof}
 Throughout the proof we assume that $a\geq3$, as the case $a=1$ is covered by Proposition \ref{singleq}.
 As before, we assume that $T$, $S$ are primitive formally dual subsets of $\ZZ_N$, with $\abs{T}<\abs{S}$, as $N$ is not a square, hence
 \[\sqrt{\vphi(N)}<\abs{T}<\sqrt{N}<\abs{S}<\frac{N}{\sqrt{\vphi(N)}}.\]
 Proposition \ref{Som} implies that $q\mid \abs{S}$. We will also show that $q\mid\abs{T}$; suppose on the contrary that $q\nmid\abs{T}$. Then $q^2\mid\bs1_S\ast\bs1_{-S}(1)$, since
 \[\abs{T(\ze_N)}^2=\frac{\abs{T}^3}{N}\bs1_S\ast\bs1_{-S}(1),\]
 contradicting Proposition \ref{Som}. %, this can only happen if $q=2$ and $\wtn{S}(1)=4$,
 %so that $N=4p^a$. Applying Theorem \ref{Fbound} and \eqref{Fmqeven} we obtain
 %\[\abs{T(\ze_N)}^2=\frac{4\abs{T}^3}{N}\leq \frac{F(N,p)^2}{4\vphi(F(N,p))}=\frac{F(N,p)}{2(1-\frac{1}{p})}=\frac{p\gcd(4,\frac{1}{2}\ord_4(p)(p^2-1))}{2(1-\frac{1}{p})}=\frac{2p^2}{p-1}.\]
 %On the other hand,
 %\[\frac{4\abs{T}^3}{N}>\frac{4\vphi(N)^{3/2}}{N}=\sqrt{2N}\bra{1-\frac{1}{p}}^{\frac{3}{2}},\]
 %whence
 %\[\sqrt{2N}=2\sqrt{2}p^{\frac{a}{2}}<\frac{2p^{\frac{7}{2}}}{(p-1)^{\frac{5}{2}}}\Longleftrightarrow \sqrt{2}p^{\frac{a-7}{2}}(p-1)^{\frac{5}{2}}<1.\]
 %The latter inequality cannot hold for $a\geq5$, because otherwise we would have $4\leq\sqrt{2}(p-1)^{3/2}<\frac{p}{p-1}\leq\frac{3}{2}$, a contradiction. So, when $a=3$, we get
 %$\sqrt{2(p-1)}<\frac{p^2}{(p-1)^2}$, and this inequality can hold only when $p=3$; for $p\geq5$ we have $2<\sqrt{2(p-1)}$ and $\frac{p^2}{(p-1)^2}<2$. Thus, $N=108$, $\vphi(N)=36$ and
 %\[6<\abs{T}<6\sqrt{3}<\abs{S}<18,\]
 %and since $2\nmid\abs{T}$, we must have $\abs{T}=9$ and $\abs{S}=12$. But then $\abs{T(\ze_N)}^2=\frac{4\cdot9^3}{108}=27$, while
 %\[\frac{F(N,3)^2}{4\vphi(F(N,3))}=\frac{2\cdot 3^2}{2}=9,\]
 %a contradiction by Theorem \ref{Fbound} and \eqref{Fmqeven}.
 
 Thus, $q$ divides both $\abs{T}$ and $\abs{S}$, hence $\abs{T}\leq p^kq$, where $a=2k+1$. This leads to
 $\abs{T}^2-\abs{T}\leq p^kq(p^kq-1)$, while $\vphi(N)=p^{2k}q(p-1)(q-1)$. Since $\abs{T}^2-\abs{T}\geq\vphi(N)$, we get $p^kq-1\geq p^k(p-1)(q-1)$; when $p$ is odd, the latter
 is $\geq2p^k(q-1)\geq p^kq$, a contradiction. Hence, $p=2$, and applying Corollary \ref{firstestimate} we obtain
 \[\wtn{T}(1)<\frac{2\abs{T}^2}{N}\leq\frac{2\cdot2^{2k}q^2}{2^{2k+1}q^2}=1,\]
 contradicting Lemma \ref{intersection}, as desired.
\end{proof}

\begin{prop}\label{p4q3}
  Let $N=p^4q^3$. Then $\ZZ_N$ has no primitive formally dual subsets.
 \end{prop}
 
 \begin{proof}
  Let $T,S\ssq\ZZ_N$ be such a pair with $\abs{T}<\abs{S}$. Since $\abs{S(\ze_N)}^2\in\ZZ$, Proposition \ref{Som} implies that $p^2q\mid\abs{S}$. Furthermore, $pq\mid\abs{T}$, since $\wtn{S}(1)$ cannot be divided
  by a cube of an integer $>1$. If $p\pdiv \abs{T}$, then we must necessarily have $p\mid\wtn{S}(1)$, which can only happen if $p=2$ or $p=3$. If $p=2$, so that $N=16q^3$, then $\abs{T}=2q^2$; the only other
  possibility would be $\abs{T}=2q$, but this contradicts Corollary \ref{firstestimate}, as $2\abs{T}^2=8q^2<N$ in this case. Hence, $\abs{S}=8q$, and since $\abs{S}>\abs{T}$, we must have $q=3$ and
  $\abs{S(\ze_N)}^2=32$ by Proposition \ref{Som}. Applying Theorem \ref{Fbound} and Propositions \ref{Fproperties} and \ref{Cpq} we get
  \[32\leq\frac{F(N,32)^2}{4\vphi(F(N,32))}=\frac{F(2^43^3,2)}{4/3}=\frac{4\gcd(3^3,2^2-1)}{4/3}=9,\]
  a contradiction. Assume next that $p=3$. If $\abs{T}=3q$, then $\abs{T}^2<\vphi(N)$, hence the only possibility that remains by Lemma \ref{mainestimate} is $\abs{T}=3q^2$ and $\abs{S}=27q$. Since
  $\abs{S}>\abs{T}$, we must have $q<9$. Furthermore, since $\abs{T}^2>\vphi(N)$ by Lemma \ref{mainestimate}, we should have
  \[9q^4>27q^2\cdot2(q-1)\Leftrightarrow q^2>6(q-1),\]
  which yields $q>4$. Lastly, since $3\mid\wtn{S}(1)$, we must have $\wtn{S}(1)=3$ by Proposition \ref{Som}, which implies $3\vphi(N)<\abs{S}^2$ by Lemma \ref{mainestimate}. Hence,
  \[3^4q^2\cdot2(q-1)<3^6q^2\Leftrightarrow2(q-1)<9,\]
  yielding $q\leq 5$, thus $q=5$, and $\abs{S(\ze_N)}^2=3^5$ by Proposition \ref{Som}. Applying Theorem \ref{Fbound} and Propositions \ref{Fproperties} and \ref{Cpq} we get
  \[3^5\leq\frac{F(N,3^5)^2}{4\vphi(F(N,3^5))}=\frac{F(N,3)}{4\cdot\frac{2}{3}\cdot\frac{4}{5}}<\frac{1}{2}\cdot3\gcd(5^3,3^4-1)=\frac{15}{2},\]
  a contradiction.
  
  Therefore, $p\pdiv \abs{T}$ cannot hold; $p^2q$ divides both $\abs{T}$ and $\abs{S}$. The only possibility is $\abs{T}=p^2q$ and $\abs{S}=p^2q^2$. The inequality $\vphi(N)<\abs{T}^2$ would then imply
  $(p-1)(q-1)<p$, which only holds when $q=2$. Applying Corollary \ref{firstestimate} we establish a contradiction, as $2\abs{T}^2=8p^4=N$, completing the proof.
 \end{proof}

Lastly, we prove the following:

\begin{prop}\label{p2q2}
 Let $N=p^2q^2$, for $p$, $q$ distinct primes. Then, $\ZZ_N$ does not have primitive formally dual subsets.
\end{prop}

\begin{proof}
 As usual, suppose that such subsets exist, with $\abs{T}\leq\abs{S}$. Without loss of generality assume $p<q$; then $\abs{T}\in\set{p,q,p^2,pq}$. If $\abs{T}=p$ or $\abs{T}=q$,
 then $\vphi(N)=pq(p-1)(q-1)\geq2q(q-1)>q^2\geq\abs{T}^2$, contradicting Lemma \ref{mainestimate} and Lemma \ref{intersection}. If $\abs{T}=p^2$, then $\abs{S}=q^2$ and \eqref{Someq}
 yields
 \[\abs{S(\ze_N)}^2=\frac{q^4}{p^2}\notin\ZZ,\]
 a contradiction. Thus, $\abs{T}=\abs{S}=pq$, and Proposition \ref{Som} applies to $T$
 as well, that is
 \[\abs{S(\ze_N)}^2=\abs{T(\ze_N)}^2=pq,\]
 or equivalently, $\wtn{T}(1)=\wtn{S}(1)=1$. Applying Lemma \ref{mainformula} for $d=N$ we obtain
 \[\frac{pq}{(pq)^{3/2}}=\frac{1}{pq(pq)^{3/2}}\ent{\abs{T(1)}^2-\abs{T(\ze_p)}^2-\abs{T(\ze_q)}^2+\abs{T(\ze_{pq})}^2},\]
 or equivalently,
 \begin{equation}\label{TT2}
  \abs{T(\ze_{pq})}^2=\abs{T(\ze_p)}^2+\abs{T(\ze_q)}^2,
 \end{equation}
 and similarly for $S$,
 \begin{equation}\label{ST2}
  \abs{S(\ze_{pq})}^2=\abs{S(\ze_p)}^2+\abs{S(\ze_q)}^2.
 \end{equation}
 Suppose first that $T(\ze_{pq})=0$; from \eqref{TT2} we also get $T(\ze_p)=T(\ze_q)=0$, hence $T(X)$ is divided by $\frac{X^{pq}-1}{X-1}=1+X+\dotsb+X^{pq-1}$. Since $T(1)=\abs{T}=pq$,
 we have
 \[T(X)\equiv 1+X+\dotsb+X^{pq-1}\bmod(X^{pq}-1).\]
 This implies that $\abs{T_{j,pq}}=1$, for all $0\leq j\leq pq-1$, therefore
 \[\wtn{T}(pq)=\wtn{T}(p^2q)=\wtn{T}(pq^2)=0,\]
 or $S(\ze_{pq})=S(\ze_p)=S(\ze_q)=0$. Moreover, $T_{j,pq}(X)$ are monomials with coefficient $1$; since $\wtn{T}(1)=1$, we have
 \[\bigcup_{j\in\ZZ_{pq}}(T_{j,pq}-T_{j-1,pq})=1+pq\ZZ_N.\]
 Taking mask polynomials on both sides we obtain
 \[X(1+X^{pq}+\dotsb+X^{pq(pq-1)})\equiv \sum_{j\in\ZZ_{pq}}T_{j,pq}(X)T_{j-1,pq}(X^{-1})\bmod(X^N-1).\]
 Differentiating both sides with respect to $X$ and then setting $X=1$ we get
 \begin{eqnarray*}
 pq+pq\frac{pq(pq-1)}{2} &\equiv& \sum_{j\in\ZZ_{pq}}T'_{j,pq}(1)T_{j-1,pq}(1)-T_{j,pq}(1)T'_{j-1,pq}(1)\bmod N\\
 &\equiv& \sum_{j\in\ZZ_{pq}}T'_{j,pq}(1)-T'_{j-1,pq}(1) \equiv0\bmod N,
 \end{eqnarray*}
 a contradiction, since the left hand side can never be divisible by $N$; if $N$ is odd, it is $pq\bmod N$, otherwise, we have (say) $p=2$, and it is $2q\bmod 2q^2$.
 
 Therefore, we assume that $T(\ze_{pq})S(\ze_{pq})\neq0$. We observe that if $T(\ze_N^p)=0$, then
 \[T(X^p)\equiv P_p(X^p)\Php+Q_p(X^p)\Phq\bmod(X^N-1),\]
 for some $P_p,Q_p\in\ZZ[X]$ with nonnegative coefficients. Since
 \[pq=T(1)=pP_p(1)+qQ_p(1),\]
 we have either $P_p(1)=q$ and $Q_p\equiv0$ or $P_p\equiv0$ and $Q_p(1)=p$. The former case implies
 \[T(X^{pq})\equiv P_p(X^{pq})\Php\bmod(X^N-1),\]
 contradicting $T(\ze_{pq})\neq0$. Thus, $P_p\equiv0$, establishing
 \[T(X^p)\equiv Q_p(X^p)\Phq\bmod(X^N-1),\]
 whence $T(\ze_{N}^{p^2})=0$. Working with $S$ instead of $T$ and $q$ instead of $p$, we get the following implications:
 \begin{claim}\label{implications}
 If $T(\ze_{pq})S(\ze_{pq})\neq0$, then
 \begin{eqnarray*}
  T(\ze_N^p)=0\Rightarrow T(\ze_N^{p^2})=0,\\
  T(\ze_N^q)=0\Rightarrow T(\ze_N^{q^2})=0,\\
  S(\ze_N^p)=0\Rightarrow S(\ze_N^{p^2})=0,\\
  S(\ze_N^q)=0\Rightarrow S(\ze_N^{q^2})=0.
 \end{eqnarray*}
 \end{claim}

 We suppose first that $T(\ze_N^p)=0$.
 Applying Lemma \ref{mainformula} for $d=p^2$ we get
 \[\abs{S(\ze_N^{q^2})}^2=\abs{S(\ze_N^{q})}^2,\]
 and another application for $d=p$ yields
 \[-\frac{1}{\sqrt{p}}\abs{T(\ze_N)}^2=-\frac{1}{q\sqrt{p}}\abs{S(\ze_q)}^2,\]
 or equivalently, $\abs{S(\ze_q)}^2=pq^2$.
 
 We distinguish two cases; first, $T(\ze_p)=0$. By \eqref{TT2} we have $\abs{T(\ze_{pq})}^2=\abs{T(\ze_q)}^2$, so Lemma \ref{mainformula} for $d=p^2q$ gives
 $\abs{S(\ze_N^q)}^2=\abs{S(\ze_N)}^2$, and thence
 \[\abs{S(\ze_N^{q^2})}^2=\abs{S(\ze_N^q)}^2=\abs{S(\ze_N)}^2=pq.\]
 Also, by $\abs{S(\ze_q)}^2=pq^2$ and \eqref{ST2} we get $\abs{S(\ze_{pq})}^2\geq pq^2$, so by Proposition \ref{basic} for $X=1$ we get
 \[p^2q^2=\frac{1}{pq}\sum_{d\mid N}\abs{S(\ze_N^d)}^2\vphi(N/d)\geq \vphi(p^2q^2)+\vphi(p^2q)+\vphi(p^2)+q(\vphi(pq)+\vphi(q))+pq=p^2q^2,\]
 therefore, since the inequality in the middle is actually an equality, we obtain
 \[S(\ze_N^p)=S(\ze_N^{p^2})=S(\ze_p)=0,\]
 and utilizing the same arguments as before (interchanging $T$ by $S$) we obtain
 \[\abs{T(\ze_N^d)}^2=\abs{S(\ze_N^d)}^2\]
 for all integers $d$. This also yields
 \begin{eqnarray*}
 \wtn{T}(1)=\wtn{T}(q)=\wtn{T}(q^2)=1,\\
 \wtn{T}(pq)=\wtn{T}(p^2q)=q,\\
 \wtn{T}(p)=\wtn{T}(p^2)=\wtn{T}(pq^2)=0.
 \end{eqnarray*}
 For every $j\in\ZZ_p$, the difference sets $T_{j,p}-T_{j-1,p}$ are subsets of $1+p\ZZ_N$; furthermore, from the equations above, every element of $1+p\ZZ_N$ occurs exactly once
 as difference $t-t'$, where $t,t'\in T$. Hence, the following equality between multisets holds:
 \[\bigcup_{j\in\ZZ_p}(T_{j,p}-T_{j-1,p})=1+p\ZZ_N,\]
 and taking mask polynomials on both sides we obtain
 \begin{equation}\label{Tjp}
 \sum_{j\in\ZZ_p}T_{j,p}(X)T_{j-1,p}(X^{-1})\equiv X\sum_{k=0}^{pq^2-1}X^{pk}\bmod(X^N-1).
 \end{equation}
 A consequence of $T(\ze_p)=0$ is $\abs{T_{j,p}}=q$ for all $j$; indeed, as
 \[T(\ze_p)=\sum_{j\in\ZZ_p}\ze_p^j\abs{T_{j,p}}=0.\]
 Differentiating the left hand side side of \eqref{Tjp} at $X=1$, we get
 \[\sum_{j\in\ZZ_p}T'_{j,p}(1)T_{j-1,p}(1)-T_{j,p}(1)T'_{j-1,p}(1)=q\sum_{j\in\ZZ_p}T'_{j,p}(1)-T'_{j-1,p}(1)=0,\]
 and differentiating the right hand side of \eqref{Tjp} at $X=1$,
 \[pq^2+p\frac{pq^2(pq^2-1)}{2}=\frac{N}{p}+N\frac{N/p-1}{2}.\]
 Next, we apply Lemma \ref{deriv}; if $N$ is odd, then the above derivative is $\equiv\frac{N}{p}\bmod N$, a contradiction. So, $N$ must be even and $p=2$, because
 this derivative is $\equiv\frac{N}{p}\bmod\frac{N}{2}$. Furthermore, $T(\ze_{q^2})=0$ implies
 \[\sum_{j\in\ZZ_{q^2}}\ze_{q^2}^j\abs{T_{j,q^2}}=0.\]
 This implies that for every $j$,
 \[\abs{T_{j,q^2}}=\abs{T_{j+q,q^2}}=\abs{T_{j+2q,q^2}}=\dotsb=\abs{T_{j+(q^2-q),q^2}},\]
 or, simply put, $\abs{T_{j,q}}=q\abs{T_{j,q^2}}$. Since $\abs{T}=2q$ and $T\neq T_{j,q}$ for any $j$, due to the primitivity of $T$, there must be $j,k\in\ZZ_q$, $j\neq k$, such that
 $T=T_{j,q}\cup T_{k,q}$; moreover, $\abs{T_{j,q}}=\abs{T_{k,q}}=q$. The differences $T-T$ taken $\bmod q$ would then be only $0$ and $\pm(j-k)$, which shows that $q=3$,
 since all possible residues $\bmod q$ appear in $\ZZ_{4q^2}^{\star}\ssq T-T$. Possibly after
 translating $T$, we may assume that $T=T_{0,3}\cup T_{1,3}$. As multisets, we would have the following equality
 \[(T_{0,3}-T_{1,3})\cup(T_{1,3}-T_{0,3})=\ZZ_{36}^{\star}.\]
 But this leads to a contradiction, as the possible differences in the left hand side are $18$, while $\ZZ_{36}^{\star}=12$.
 
 The next case is $T(\ze_p)\neq0$. Applying Lemma \ref{mainformula} for $d=p^2q$ and \eqref{TT2} we obtain
 \[-\frac{1}{p\sqrt{q}}\abs{T(\ze_p)}^2=\frac{1}{\sqrt{q}}\ent{\abs{S(\ze_N^q)}^2-\abs{S(\ze_N)}^2},\]
 hence $\abs{S(\ze_N^q)}^2<\abs{S(\ze_N)}^2$ or equivalently, $\wtn{T}(q)<1$. Therefore, $S(\ze_N^q)=0$, and by Lemma \ref{mainformula} for $d=p^2$ we also get
 \[S(\ze_N^q)=S(\ze_N^{q^2})=0.\]
 Applying Lemma \ref{mainformula} for $d=pq^2$ and $d=q^2$, we obtain the formulae
 \[\frac{1}{q}\ent{-\abs{T(\ze_q)}^2-\abs{T(\ze_N^{q^2})}^2+\abs{T(\ze_N^q)}^2}=\abs{S(\ze_N^p)}^2-\abs{S(\ze_N)}^2\]
 and
 \[\frac{1}{q}\ent{\abs{T(\ze_N^{q^2})}^2-\abs{T(\ze_N^q)}^2}=\frac{1}{p}\ent{\abs{S(\ze_N^{p^2})}^2-\abs{S(\ze_N^p)}^2}.\]
 Adding these equations by parts yields
 \begin{eqnarray*}
 -\frac{1}{q}\abs{T(\ze_q)}^2 &=& \frac{1}{p}\ent{\abs{S(\ze_N^{p^2})}^2-\abs{S(\ze_N^p)}^2}+\ent{\abs{S(\ze_N^p)}^2-\abs{S(\ze_N)}^2}\\
 &=& \frac{1}{p}\ent{\abs{S(\ze_N^{p^2})}^2-\abs{S(\ze_N)}^2}+\bra{1-\frac{1}{p}}\ent{\abs{S(\ze_N^p)}^2-\abs{S(\ze_N)}^2}.
 \end{eqnarray*}
 Since the left hand side is negative, either one of $\abs{S(\ze_N^p)}^2$ and $\abs{S(\ze_N^{p^2})}^2$ must be less than $\abs{S(\ze_N)}^2$; but this means that one of them is zero,
 otherwise, we would either have $\wtn{T}(p)<1$ or $\wtn{T}(p^2)<1$. If $S(\ze_N^p)=0$, then $S(\ze_N^{p^2})=0$ as well by Claim \ref{implications}, since $S(\ze_{pq})\neq0$, 
 so at any rate, $S(\ze_N^{p^2})=0$. Hence,
 \[-\frac{1}{q}\abs{T(\ze_q)}^2 = \bra{1-\frac{1}{p}}\abs{S(\ze_N^p)}^2-\abs{S(\ze_N)}^2.\]
 If $\wtn{T}(p)\geq2$, then $\abs{S(\ze_N^p)}^2\geq2pq$ and the right hand side would be $\geq2(p-1)q-pq=(p-2)q\geq0$, while the left hand side is negative. so, either $\wtn{T}(p)=0$
 or $1$. If $\wtn{T}=1$, then $\abs{T(\ze_q)}^2=q^2$, an absurdity as $pq\mid\abs{T(\ze_N^d)}^2$ for all $d\in\ZZ$. This shows that $S(\ze_N^p)=0$ as well; a symmetric argument also
 yields $T(\ze_N^q)=T(\ze_N^{q^2})=0$. Now consider the difference sets $T_{j,p}-T_{k,p}$, for $j\not\equiv k\bmod p$; all differences are prime to $p$, and since 
 $\wtn{T}(q)=\wtn{T}(q^2)=0$, we must have the following equality of multisets:
 \[\bigcup_{\substack{j,k\in\ZZ_p\\ j\not\equiv k\bmod p}}(T_{j,p}-T_{k,p})=\ZZ_N^{\star}.\]
 Taking mask polynomials, we get
 \[\sum_{\substack{j,k\in\ZZ_p\\ j\not\equiv k\bmod p}}T_{j,p}(X)T_{k,p}(X^{-1})\equiv R_N(X)\bmod(X^N-1),\]
 whence for $X=1$,
 \[\sum_{\substack{j,k\in\ZZ_p\\ j\not\equiv k\bmod p}}\abs{T_{j,p}}\abs{T_{k,p}}=pq(p-1)(q-1).\]
 The left hand side is also equal to
 \[\ent{\sum_{j\in\ZZ_p}\abs{T_{j,p}}}^2-\sum_{j\in\ZZ_p}\abs{T_{j,p}}^2=p^2q^2-\sum_{j\in\ZZ_p}\abs{T_{j,p}}^2.\]
 therefore,
 \[\sum_{j\in\ZZ_p}\abs{T_{j,p}}^2=pq(p+q-1).\]
 On the other hand, using a simliar argument as before, the fact that $T(\ze_N^{q^2})=0$ implies $p\mid\abs{T_{j,p}}$ for all $p$, whence $p^2\mid\sum_{j\in\ZZ_p}\abs{T_{j,p}}^2$;
 we should also have $q^2\mid \sum_{j\in\ZZ_p}\abs{T_{j,p}}^2$, as $T(\ze_N^{p^2})=0$ as well, which is clearly an absurdity as $p^2q^2\nmid pq(p+q-1)$.
 
 Thus, we may assume that
 \[T(\ze_N^p)T(\ze_N^q)T(\ze_{pq})S(\ze_N^p)S(\ze_N^q)S(\ze_{pq})\neq0,\]
 otherwise we would revisit one of the previous cases, possibly by interchanging $S$ by $T$ and $p$ by $q$. We will show that
 \begin{equation}\label{qTp}
 q\wtn{T}(p^2)+\wtn{T}(p^2q)\geq q,
 \end{equation}
 and similarly,
 \begin{equation}\label{pTq}
 p\wtn{T}(q^2)+\wtn{T}(pq^2)\geq p.
 \end{equation}
 If $\wtn{T}(p^2)\geq1$, \eqref{qTp} is trivially satisfied, so assume $\wtn{T}(p^2)=0$, or equivalently, $S(\ze_N^{p^2})=0$. Then,
 \[S(X^{p^2})\equiv Q(X^{p^2})\Phq\bmod(X^N-1),\]
 for some $Q(X)\in\ZZ[X]$ with nonnegative coefficients. Therefore,
 \[S(X^{p^2q})\equiv qQ(X^{p^2q})\bmod(X^N-1),\]
 whence
 \[\abs{S(\ze_q)}^2=q^2\abs{Q(\ze_q)}^2.\]
 The term $\abs{Q(\ze_q)}^2$ is simultaneously an algebraic integer by definition and a rational number due to the above equation, hence an integer. This implies $q^2\mid\abs{S(\ze_q)}^2$,
 hence $q\mid\wtn{T}(p^2q)$; we note that $\wtn{T}(p^2q)$ and $\wtn{T}(p^2)$ cannot be both zero, otherwise $q^2\mid\abs{S}$, a contradiction. Thus, $\wtn{T}(p^2q)\geq q$ in this case,
 proving \eqref{qTp} and \eqref{pTq} at all cases. Now apply Proposition \ref{basic} for $X=1$, along with \eqref{qTp} and \eqref{pTq} we get
 \begin{eqnarray*}
 \abs{T}^2 &=& \sum_{d\mid N}\wtn{T}(d)\vphi(N/d)\\
 &\geq& pq+\vphi(p^2q^2)+\vphi(p^2q)+\vphi(pq^2)+\vphi(pq)+\\
 &+&\wtn{T}(p^2)\vphi(q^2)+\wtn{T}(p^2q)\vphi(q)+\wtn{T}(q^2)\vphi(p^2)+\wtn{T}(pq^2)\vphi(p)\\
 &\geq& pq+pq(p-1)(q-1)+q(p-1)(q-1)+p(p-1)(q-1)+\\
 &+& (p-1)(q-1)+q(q-1)+p(p-1)\\
 &=& p^2q^2+(p-1)(q-1)>p^2q^2,
 \end{eqnarray*}
 which is clearly a contradiction, concluding the proof.
\end{proof}

Summarizing the results of this section, we conclude that the field descent method tackles the cases $p^kq^l$ with both exponents relatively high, while the polynomial method tackles the cases where
one exponent is small. Combining these methods, we see that for most pairs of primes the conjecture is settled; of course, this needs to be properly quantified. In the Appendix, we find the number of
exceptions among all $N$ with $\rad(N)=pq$ and $p,q<10^3$, just to get an idea. This number is significantly lower if we use the polynomial method to prove the $N=p^4q^2$ case.

 \bigskip
 \section{Beyond two prime factors}\label{beyond}%Orders divisible exactly by a prime}
 \bigskip
 
 The main obstruction to apply the field descent method when $\om(N)>2$ is the fact that primitivity of $T\ssq\ZZ_N$ does not imply $(T-T)\cap\ZZ_N^{\star}\neq\vn$, as can be seen from the Remark
 immediately after Lemma \ref{intersection}. However, it needs to be emphasized that so far this method was only applied to the condition $\abs{S(\ze_N)}^2\in\ZZ$, and not
 $\abs{S(\ze_N^d)}^2\in\ZZ$ in general. For $S(\ze_N^d)$, the constant $C$ in Theorem \ref{Fbound} can be as large as $d$, so it might suffice to show that there always exists some $d\mid N$,
 say $d\leq N^{1/5}$, satisfying $S(\ze_N^d)\neq0$. Moreover, Lemma \ref{intersection} was proven for primitive sets, without any other condition. It is possible that this can be extended to
 other cases with conditions such as $\abs{T}\mid N$, or the requirement that the differences $t-t'$ with $t,t'\in T$ are equidistributed in every divisor class $d\ZZ_N^{\star}$. 
 
 Besides the case for square--free $N$ where Conjecture \ref{mainconj} has been confirmed \cite{XPC}, we will show that \ref{mainconj} is also true for another family of orders $N$ that satisfy the so--called
 \emph{self conjugacy} property with respect to a prime factor $p$ (there is no restriction on the number of distinct prime factors for such $N$). This notion was first used by Turyn \cite{Turyn65} to attack
 the circulant Hadamard conjecture stated by Ryser \cite{CircHadamard}.%, and then extended by Schmidt \cite{Schmidt99,SchmidtFDM}.
 
 \begin{defn}
  A prime $p$ is called self--conjugate $\bmod N$ if every ideal $\mf P\ssq\ZZ[\ze_N]$ dividing $p\ZZ[\ze_N]$ is invariant under complex conjugation, i.e. $\overline{\mf P}=\mf P$.
 \end{defn}
 
 In other words, the complex conjugation belongs to the decomposition group of any prime ideal $\mf P\mid p$. A characterization of the decomposition group (cf. Theorem 1.4.3 \cite{SchmidtFDM})
 shows that $\sig\in G_{\mf P}$ if and only if $\sig(\ze_m)=\ze_m^{p^j}$ for some $j\in \ZZ$, where $N=p^am$, $p\nmid m$ (i.e. $m$ is the $p$--free part of $N$). From this follows the result of
 Turyn \cite{Turyn65} (see also Corollary 1.4.5 \cite{SchmidtFDM}), a weaker version of which we state below.
 
 \begin{thm}\label{selfconj}
  Let $A\in\ZZ[\ze_N]$ such that $\abs{A}^2\equiv0\bmod p^{2b}$, where $p$ is self--conjugate $\bmod N$. Then $A\equiv0\bmod p^b\ZZ[\ze_N]$.
 \end{thm}

 This the main result of this section.

 \begin{thm}\label{pN}
 Let $p$ be a prime such that $p\pdiv N$ and $p^j\equiv-1\bmod\frac{N}{p}$ for some integer $j$. Then $\ZZ_N$ does not have any pair of primitive formally dual subsets.
 \end{thm}
 \begin{proof}
 The hypothesis clearly shows that $p$ is self--conjugate $\bmod N$.
  Let $T$, $S$ be a pair of primitive formally dual subsets of $\ZZ_N$.
 Without loss of generality, we assume $p\mid \abs{T}$, so that
 $p\nmid \abs{S}$. For every $d\mid N$ we have
 \[\abs{T(\ze_N^d)}^2=\bs 1_S\ast\bs1_{-S}(d)\frac{\abs{T}^2}{\abs{S}},\]
 hence $p^2\mid \abs{T(\ze_N^d)}^2$. We consider the mask polynomial $T(X)\bmod (p,X^{N/p}-1)$, and let $\mathfrak{P}$ be any prime ideal in $\QQ(\ze_N^p)$ that is above $p$. The degree
 of the residue field extension
 \[f=[\ZZ[\ze_N^p]/\mf{P}:\ZZ/p\ZZ]\]
 is also equal to the multiplicative order of $p\bmod \frac{N}{p}$. In particular, the ring epimorphism
 \[\ZZ[\ze_N^p]\twoheadrightarrow \ka(\mf{P}):=\ZZ[\ze_N^p]/\mf{P}\]
 sends all $\frac{N}{p}$th roots of unity of $\CC$ to the $\frac{N}{p}$th roots of unity of $\ka(\mf{P})$. Let $\ol{T}(X)$ be the image of $T(X)$ under the projection
 \[\ZZ[X]\twoheadrightarrow \FF_p[X].\]
 Since $p^2\mid \abs{T(\ze_N^d)}^2$ for every $d\mid N$, we must have
 \[T(\ze_N^d)\equiv 0\bmod\mf{P},\]
 for every $d\mid N$ by Theorem \ref{selfconj}; restricting to $p\mid d$, we observe that $\ol{T}(X)$ accepts as roots all $\frac{N}{p}$th roots of unity of $\ka(\mf{P})\cong\FF_{p^f}$, which yields
 \[\ol{T}(X)\equiv 0\bmod(X^{N/p}-1).\]
 Lifting up to $\ZZ[X]$, we obtain
 \[T(X)\equiv pQ(X)\bmod(X^{N/p}-1),\]
 or equivalently,
 \begin{equation}\label{pT}
 T(X^p)\equiv pQ(X^p)\bmod (X^N-1),
 \end{equation}
 where we can take $Q(X)\in\ZZ_{\geq0}[X]$.
 But $T(X^p)\bmod (X^N-1)$ is also the mask polynomial of the multi-set $p\cdot T$; the multiplicities that appear in this multi-set are at most $p$, since $T$ is a proper set. On the other
 hand, \eqref{pT} shows that all multiplicities are at least $p$. This can only occur when $t\in T$ implies $t+jN/p\in T$ for all $j$, as these elements exactly have the same image under
 multiplication by $p$. Therefore, $T$ is a union of $p$-cycles, in particular,
 \[T(X)\equiv \Php R(X)\bmod (X^N-1),\]
 for some $R(X)\in\ZZ_{\geq0}[X]$. Then, for every $t\in T$, $t-N/p\in T$, which shows that
 \[\bs1_{T}\ast\bs1_{-T}(\tfrac{N}{p})=\abs{T},\]
 hence by \eqref{formdualalt}, $\abs{S(\ze_p)}=\abs{S}$, contradicting Proposition \ref{upperboundmask}.
 %Define $S_j:=\set{s\in S : s\equiv j\bmod p}$. By triangle inequality on the complex plane we get
 %\[\abs{S(\ze_p)}=\abs{\sum_{j=0}^{p-1}\ze_p^j\abs{S_j}}\leq\sum_{j=0}\abs{S_j}=\abs{S},\]
 %where equality occurs precisely when all $\ze_p^j\abs{S_j}$ have the same direction, considered as vectors in a two-dimensional vector space. This can only happen if $S=S_j$ for some
 %$j$; but then $S-S\ssq p\ZZ_N$, thus $S$ is not primitive, contradiction. 
 \end{proof}
 
 %We remind a result of Niven \cite{Niven}:
 
% \begin{thm}[Corollary 1, \cite{Niven}]\label{density}
%  Let $A\ssq\ZZ$, and denote by $A_p$ the subset of elements of $A$ divisible by $p$ but not $p^2$, for $p$ prime. Suppose that for a set of primes $P$ we have
%  $\sum_{p\in P}\frac{1}{p}=\infty$ and $d(A_p)=0$ for every $p\in P$, where $d(B)$ is the density of a set $B$ of integers. Then $d(A)=0$. 
% \end{thm}
 
% Let $A$ be the set of integers, such that
% \[N\in A\Longleftrightarrow \ZZ_N \text{ has a pair of primitive formally dual subsets}.\]
% Theorem \ref{pN} implies that $A_p=\vn$, for every prime $p$. By Theorem \ref{density} we obtain:
 
% \begin{cor}
%  The set $A$ has density zero.
% \end{cor}

% With the above corollary, we know that the conjecture of Cohn, Kumar, Reiher, and Sch\"urmann is true except in a set of density zero.

 \bigskip
 \appendix
 
 \section{Products of two powers of small primes}
 
 We will focus on $N=p^kq^l$, where $p,q<10^3$. As mentioned at the end of Section \ref{twoprimes}, the field descent method tackles the cases where $k+l$ is large, and the polynomial method tackles those
 where $k+l$ is small, roughly speaking. Most of the times there is no gap, and when there is, it usually consists of a single exception. This search for exceptions is assisted by simple computer programs
 on wxMaxima\footnote{Available at \url{https://sites.google.com/site/romanosdiogenesmalikiosis/computational-data}}. In particular, for every pair $(p,q)$ with $p<q<10^3$ these programs compute
 $\nu_p(q^{p-1}-1)$ and $\nu_q(p^{q-1}-1)$ when $p>2$; when $p=2$ they compute $\nu_2(\frac{1}{2}\ord_4(q)(q^2-1))$, as well the number of possible exceptions from each pair of the form $(2,q)$.
 
 Before proceeding, we will need two useful
 propositions for small primes, as well as the notions of a \emph{Wieferich prime} and a \emph{Wieferich pair} \cite{Wieferich}:
 
 \begin{defn}\label{Wieferich}
  A prime $p$ is called Wieferich, if $p^2\mid 2^{p-1}-1$. A pair of primes $(p,q)$ is called a Wieferich pair, if $p^2\mid q^{p-1}-1$ and $q^2\mid p^{q-1}-1$.
 \end{defn}
 
 There are only two known Wieferich primes, namely 1193 and 3511, and only 7 known Wieferich pairs (sequences A124121 and A124122 from OEIS\footnote{\url{ http://oeis.org/A124121} and \url{ http://oeis.org/A124122}}):
 %\begin{equation}\label{Wieferichpairs}
 \[
 (2,1093), (3,1006003), (5,1645333507), (5,188748146801), (83,4871), (911,318917), (2903,18787).
 \]
 %\end{equation}
 
 \begin{prop}\label{p3q3}
  Let $N=p^3q^3$. If $\ZZ_N$ has a pair of primitive formally dual subsets, then $p$ and $q$ are simultaneously twin primes and a Wieferich pair.
 \end{prop}
 
 \begin{proof}
  Let $T,S\ssq\ZZ_N$ be such a pair with $\abs{T}<\abs{S}$. Since $\abs{S(\ze_N)}^2\in\ZZ$, Proposition \ref{Som} implies that $pq\mid\abs{S}$. Furthermore, since $\wtn{S}(1)$ cannot be divisible by
  the cube of any integer $>1$ by Proposition \ref{Som}, we must also have $pq\mid\abs{T}$. Without loss of generality, we consider $p<q$, so by Proposition \ref{STN}
  the only possibility for $\abs{T}$ and $\abs{S}$ is
  \[\abs{T}=p^2q,\	\abs{S}=pq^2.\]
  If $N$ is even, then $p=2$, hence by Corollary \ref{firstestimate} we get $32q^2=2\abs{T}^2>N=8q^3$, hence $q=3$, and $\abs{S(\ze_N)}^2=27$ by Proposition \ref{Som}. Applying 
  Theorem \ref{Fbound} and Propositions \ref{Fproperties} and \ref{Cpq}, we
  obtain
  \[27\leq\frac{F(216,27)^2}{4\vphi(F(216,27))}=18,\]
  a contradiction by Proposition \ref{Cpq}.
  
  So, $N$ must be odd. By Lemma \ref{mainestimate}, we get
  \[pq\sqrt{(p-1)(q-1)}<p^2q,\]
  or equivalently, $(p-1)(q-1)<p^2$, which cannot hold true unless $q=p+2$, i.e. $p$ and $q$ are twin primes. Since $\abs{S}=pq^2$, we will have $\abs{S(\ze_N)}^2=q^3$, so applying again Theorem \ref{Fbound}
  and Propositions \ref{Fproperties} and \ref{Cpq} we get
  \[q^3\leq\frac{F(p^3q^3,q^3)^2}{4\vphi(F(p^3q^3,q^3))}=\frac{F(p^3q^3,q)}{4(1-\frac{1}{p})(1-\frac{1}{q})}=\frac{q\gcd(p^3,q^{p-1}-1)}{4(1-\frac{1}{p})(1-\frac{1}{q})},\]
  or equivalently, $4(p-1)(q-1)q\leq p\gcd(p^3,q^{p-1}-1)$. This inequality can only hold when $p^3\mid q^{p-1}-1$. We note that this cannot happen when $(p,q)=(3,5)$ or $(5,7)$, so we may assume that $p\geq7$.
  Next, we examine
  \[\abs{T(\ze_N)}^2=\frac{\abs{T}^3}{N}\wtn{S}(1)=p^3\wtn{S}(1).\]
  By Lemma \ref{mainestimate} we get
  \[\wtn{S}(1)<\frac{\abs{S}^2}{\vphi(N)}=\frac{q^2}{(p-1)(q-1)}=1+\frac{4p+5}{p^2-1}\leq 1+\frac{p(p-2)}{(p-1)(p+1)}<2,\]
  since $q=p+2$, so $\wtn{S}(1)=1$ and $\abs{T(\ze_N)}^2=p^3$. Applying the field descent method, as was done with $\abs{S(\ze_N)}^2$, we also get $q^3\mid p^{q-1}-1$, thus $p$ and $q$ form a Wieferich pair.
 \end{proof}
 
 We remark that $p$ and $q$ satisfy a stronger condition than the one given in Definition \ref{Wieferich}; at any rate, these conditions never hold when $p,q<10^3$. It is possible, however,
 that an application of the polynomial method (that was left out from the proof) would eventually show that no such $\ZZ_N$ has a primitive formal dual pair.

 \begin{prop}\label{pmq3}
  Let $N=p^mq^3$, with $m\geq5$ and $p,q<10^3$. Then $\ZZ_N$ has no primitive formally dual subsets.
 \end{prop}
 
 \begin{proof}
  Let $T,S\ssq\ZZ_N$ be such a pair of subsets with $\abs{S}>\abs{T}$. As before, Proposition \ref{Som} yields $q\mid\abs{S}$ and $q\mid\abs{T}$. We distinguish two cases.
  
  \noindent
  $\boxed{q\pdiv\abs{S}}$ As $\abs{S(\ze_N)}^2=\frac{\abs{S}^3}{N}$, this must be equal to a power of $p$; furthermore, $\abs{S(\ze_N)}^2>\sqrt{N}\geq p^{5/2}q^{3/2}$. Applying Theorem \ref{Fbound} and
  Proposition \ref{Fproperties}, we get
  \begin{equation}\label{p5q3}
  p^{5/2}q^{3/2}<\frac{F(N,\abs{S(\ze_N)}^2)^2}{4\vphi(F(N,\abs{S(\ze_N)}^2))}=\frac{F(N,p)}{4(1-\frac{1}{p})(1-\frac{1}{q})}.
  \end{equation}
  If $p=2$, then $F(N,2)\leq4\gcd(q^3,2^{q-1}-1)=4q$, by Proposition \ref{Cpq} and the fact that no prime $q<10^3$ is Wieferich. Therefore, \eqref{p5q3} gives $2\sqrt{2}(q-1)<\sqrt{q}$, which never holds.
  If $q=2$, then $F(N,p)\leq p\gcd(8,p^2-1)=8p$, by Proposition \ref{Cpq}, since $p$ is odd. In this case, \eqref{p5q3} gives $\sqrt{p}(p-1)<\sqrt{2}$, which again never holds. So, we assume that $N$ is odd.
  Then, \eqref{p5q3} becomes
  \[p^{5/2}q^{3/2}<\frac{p\gcd(q^3,p^{q-1}-1)}{4(1-\frac{1}{p})(1-\frac{1}{q})}<\frac{1}{2}pq^a,\]
  where $a\leq3$. This in turn yields
  \begin{equation}\label{p5q3a}
  2p^{3/2}<q^{a-3/2},
  \end{equation}
  and as the right hand side is $<q^{3/2}$, we must certainly have $p<q$. A simple search for primes $p<q<10^3$ reveals that we never have $q^3\mid p^{q-1}-1$, so $a\leq2$. If $a=1$, \eqref{p5q3a} cannot
  hold; the only pairs of primes $(p,q)$ with $p<q<10^3$ and $q^2\mid p^{q-1}-1$ are
  \begin{equation}\label{qa2}
  (3,11), (11,71), (13,863), (19,137), (71,331), (127,907).
  \end{equation}
  However, none of them satisfies $2p\sqrt{p}<\sqrt{q}$, therefore we cannot have $q\pdiv\abs{S}$.
  
  \noindent
  $\boxed{q\pdiv\abs{T}}$ Then $\frac{\abs{T}^3}{N}$ is a power of $p$. For $p,q\geq7$, we have
  \[\wtn{S}(1)<\frac{\abs{S}^2}{\vphi(N)}<\frac{N^2}{\vphi(N)^2}=\bra{\frac{pq}{(p-1)(q-1)}}^2\leq \bra{\frac{77}{60}}^2<2,\]
  and if $p,q\geq5$ with $p,q\neq 7$, then 
  \[\wtn{S}(1)<\bra{\frac{pq}{(p-1)(q-1)}}^2\leq \bra{\frac{55}{40}}^2<2,\]
  so in both cases,
  \[\abs{T(\ze_N)}^2=\frac{\abs{T}^3}{N}.\]
  In these cases, we get
  \[\frac{\abs{T}^3}{N}\leq\frac{F(N,\abs{T(\ze_N)}^2)^2}{4\vphi(F(N,\abs{T(\ze_N)}^2))}=\frac{F(N,p)}{4(1-\frac{1}{p})(1-\frac{1}{q})}=\frac{p\gcd(q^3,p^{q-1}-1)}{4(1-\frac{1}{p})(1-\frac{1}{q})}.\]
  By Lemma \ref{mainestimate}, we get
  \[\frac{\abs{T}^3}{N}>\frac{\vphi(N)^{3/2}}{N}=\sqrt{N}\ent{(1-\tfrac{1}{p})(1-\tfrac{1}{q})}^{\frac{3}{2}}.\]
  Combining the above, we get
  \[p^{5/2}q^{3/2}\leq \sqrt{N}<pq^a,\]
  since $4[(1-\frac{1}{p})(1-\frac{1}{q})]^{5/2}>1$ when $p,q\geq 5$, where $q^a=\gcd(q^3,p^{q-1},1)$. The above is equivalent to $p\sqrt{p}<q^{a-3/2}$, and as the latter is $\leq q\sqrt{q}$, we must have
  $p<q$. Again, with a simple computer search we find that we cannot have $a=3$ for such primes with $p<q<10^3$; moreover, if $a=1$, $p\sqrt{p}<q^{a-3/2}$ cannot hold, so we must have $a=2$ and
  $p^3<q$. We are led again to the pairs in \eqref{qa2}, but none satisfies $p^3<q$. The only case that is not tackled for $p,q\geq 5$ is when we have the pair $(5,7)$, or equivalently, $35\mid N$. In that
  case, by Theorem \ref{Fbound} and Propositions \ref{Fproperties} and \ref{Cpq} we get
  \[5^{5/2}7^{3/2}\leq\sqrt{N}<\abs{S(\ze_N)}^2\leq\frac{F(N,\abs{S(\ze_N)}^2)^2}{4\vphi(F(N,\abs{S(\ze_N)}^2))}\leq\frac{F(N,35)}{4\cdot\frac{24}{35}}=\frac{5^37^2}{96},\]
  since $5^2\pdiv 7^4-1$ and $7\pdiv5^6-1$. But this leads to $96<\sqrt{35}$ a contradiction.
  
  Next, assume that $N$ is odd and $p=3$. Then $\sqrt{N}\geq3^{5/2}q^{3/2}$ and
  \[\sqrt{N}<\abs{S(\ze_N)}^2\leq\frac{F(N,\abs{S(\ze_N)}^2)^2}{4\vphi(F(N,\abs{S(\ze_N)}^2))}\leq\frac{F(N,3q)}{\frac{8}{3}(1-\frac{1}{q})},\]
  yielding
  \begin{equation}\label{Np3}
  8\cdot3^{3/2}\sqrt{q}(q-1)<\gcd(3^mq^3,(3^{q-1}-1)(q^2-1)),
  \end{equation}
  by Theorem \ref{Fbound} and Propositions \ref{Fproperties} and \ref{Cpq}. If $q\neq11$ and $q<10^3$, then the right hand side is $q\gcd(3^m,q^2-1)\leq3^{\nu_3(q^2-1)}$. The left hand side is
  $>3^3$, as $8>3^{3/2}$ and $q-1>\sqrt{q}$ for all odd primes $q$; hence, $\nu_3(q^2-1)\geq4$, therefore $q\geq163$. At any rate, $\nu_3(q^2-1)\leq5$ for $q<10^3$, yielding
  \[\frac{q-1}{\sqrt{q}}<\frac{3^{7/2}}{8},\]
  which is a contradiction, as $6>\frac{3^{7/2}}{8}$ and $\frac{q-1}{\sqrt{q}}>12$ for $q\geq 163$. If $q=11$, \eqref{Np3} becomes
  \[8\cdot3^{3/2}\frac{10}{\sqrt{11}}<3\cdot11^2,\]
  which holds. Applying Theorem \ref{Fbound} and Propositions \ref{Fproperties} and \ref{Cpq} on $\abs{S(\ze_N)}^2>\sqrt{N}$ gives
  \[3^{5/2}11^{3/2}\leq\sqrt{N}<\frac{F(N,33)^2}{4\vphi(F(N,33))}=\frac{3\cdot11^2}{\frac{80}{33}}\]
  or $80\sqrt{3}<11\sqrt{11}$, a contradiction.
  
  If $N$ is odd and $q=3$, then applying Theorem \ref{Fbound} and Propositions \ref{Fproperties} and \ref{Cpq} for $\abs{S(\ze_N)}^2>\sqrt{N}\geq 3^{3/2}p^{5/2}$,
  \[3^{3/2}p^{5/2}<\frac{F(N,3p)^2}{4\vphi(F(N,3p))}\]
  holds, or equivalently,
  \[8\sqrt{3}p^{3/2}(p-1)<\gcd(p^m3^3,(p^2-1)(3^{p-1}-1)).\]
  If $p\neq11$ and $p<10^3$, then the right hand side is $p\gcd(27,p^2-1)\leq27p$, so we must have
  \[8\sqrt{p}(p-1)<9\sqrt{3}<16,\]
  whence $\sqrt{p}(p-1)<2$, a contradiction, as $p\geq3$ does not satisfy this inequality. If $p=11$, then we get $8\cdot11^{3/2}\sqrt{3}\cdot 10<3\cdot 11^2$ or $80<\sqrt{33}$, a contradiction.
  
  Lastly, suppose that $N$ is even. Assume first that $p=2$, so that $\abs{S(\ze_N)}^2>\sqrt{N}=2^{m/2}q^{3/2}$. By Theorem \ref{Fbound} and Propositions \ref{Fproperties} and \ref{Cpq} we get
  \[2^{m/2}q^{3/2}<\frac{F(N,\abs{S(\ze_N)}^2)}{4\vphi(F(N,\abs{S(\ze_N)}^2))}\leq\frac{F(N,2q)}{2(1-\frac{1}{q})},\]
  yielding
  \[2^{\frac{m}{2}+1}\sqrt{q}(q-1)<\gcd(2^mq^3,(2^{q-1}-1)(\tfrac{1}{2}\ord_4(q)(q^2-1)))=q\gcd(2^m,\tfrac{1}{2}\ord_4(q)(q^2-1)),\]
  as no prime $q<10^3$ is Wieferich. We put $2^a=\gcd(2^m,\tfrac{1}{2}\ord_4(q)(q^2-1))$; for $q<10^3$, $a\leq8$ holds. Furthermore, the above inequality yields
  \begin{equation}\label{Np2}
  q-1<2^{a-\frac{m}{2}-1}\sqrt{q}.
  \end{equation}
  By definition, $a\leq m$, so the right hand side is $\leq 2^{\frac{a}{2}-1}\sqrt{q}\leq 8\sqrt{q}$, as 
  \[\nu_2(\tfrac{1}{2}\ord_4(q)(q^2-1))\leq8\]
  when $q<10^3$. The inequality $q-1<8\sqrt{q}$ further yields $q\leq67$; for these primes,
  \[\nu_2(\tfrac{1}{2}\ord_4(q)(q^2-1))\leq6,\]
  which in turn implies $q-1<4\sqrt{q}$, therefore $q\leq17$. Repeating this argument once again, we deduce
  \[\nu_2(\tfrac{1}{2}\ord_4(q)(q^2-1))\leq4.\]
  Since $m\geq5$, \eqref{Np2} gives
  \[q-1<\sqrt{2q},\]
  which can only hold for $q=3$. However, $\nu_2(\tfrac{1}{2}\ord_4(3)(3^2-1))=\nu_2(8)=3$, so substituting at \eqref{Np2} we obtain $q-1<\sqrt{\frac{q}{2}}$, a contradiction.
  The last case is $N$ even and $q=2$. As before, we apply Theorem \ref{Fbound} and Propositions \ref{Fproperties} and \ref{Cpq} on $\abs{S(\ze_N)}^2$, getting
  \[p^{m/2}2^{3/2}=\sqrt{N}<\abs{S(\ze_N)}^2\leq\frac{F(N,\abs{S(\ze_N)}^2)^2}{4\vphi(F(N,\abs{S(\ze_N)}^2))}\leq\frac{F(N,2p)}{2(1-\frac{1}{p})},\]
  or equivalently,
  \[p^{\frac{m}{2}-1}2^{5/2}(p-1)<\gcd(2^3p^m,\tfrac{1}{2}\ord_4(p)(p^2-1)(2^{p-1}-1))=p\gcd(8,\tfrac{1}{2}\ord_4(p)(p^2-1))\leq8p,\]
  as no prime $p<10^3$ is Wieferich. The above gives
  \[p^{\frac{m}{2}-2}(p-1)<\sqrt{2},\]
  a contradiction, wince the left hand side is $\geq \sqrt{p}(p-1)>2$, concluding the proof.
 \end{proof}

 \begin{prop}\label{p4q4}
  Let $N=p^mq^n$, with $m,n\geq4$ and $p,q<10^3$. Then $\ZZ_N$ has no primitive formally dual subsets.
 \end{prop}
 
 \begin{proof}
  Without loss of generality, we assume $p<q$. Also, we suppose that $T,S\ssq\ZZ_N$ is a primitive pair of formally dual sets with $\abs{S}\geq\abs{T}$. Then by Proposition \ref{Som},
  \[\abs{S(\ze_N)}^2=\frac{\abs{S}^3}{N}\geq\sqrt{N}\geq p^2q^2.\]
  Applying Theorem \ref{Fbound} and Propositions \ref{Fproperties} and \ref{Cpq} we obtain
  \begin{equation}\label{mn4}
  p^2q^2\leq\frac{F(N,\abs{S(\ze_N)}^2)^2}{4\vphi(F(N,\abs{S(\ze_N)}^2))}\leq\frac{F(N,pq)}{4(1-\frac{1}{p})(1-\frac{1}{q})}.
  \end{equation}
  Suppose first that $N$ is odd; for $p,q<10^3$, we have either $F(N,pq)=p^aq$ for some $a\leq5$ or $F(N,pq)=pq^2$. In the latter case, \eqref{mn4} yields
  \[4(1-\tfrac{1}{p})(1-\tfrac{1}{q})p<1,\]
  a contradiction, as the left hand side is $>2p$. In the former case, we get
  \[4(1-\tfrac{1}{p})(1-\tfrac{1}{q})q<p^{a-2}\]
  which implies
  \begin{equation}\label{Nodd}
  2q<p^{a-2},
  \end{equation}
  since $4(1-\tfrac{1}{p})(1-\tfrac{1}{q})\geq4\cdot\frac{8}{15}>2$ when $p$, $q$ are odd primes. Therefore, either $a=4$ or $a=5$; both cases occur only when $q\geq163$. The case $a=5$ happens only for $p=3$;
  \eqref{Nodd} cannot hold, as $326\leq 2q<27$ is false. Moreover, $a=4$ happens for $p\leq 13$, but again \eqref{Nodd} fails, as $p^{2}\leq 169<326\leq 2q$.
  
  Next, suppose that $N$ is even, that is $p=2$. \eqref{mn4} gives
  \begin{equation}\label{pis2}
  8q(q-1)\leq F(N,2q)=\gcd(2^mq^n,\tfrac{1}{2}\ord_4(q)(q^2-1)(2^{q-1}-1))=2^aq,
  \end{equation}
  since no prime $q<10^3$ is Wieferich; here, $a=\min(m,\nu_2(\tfrac{1}{2}\ord_4(q)(q^2-1)))$. \eqref{pis2} is equivalent to
  \begin{equation}\label{Neven}
   q-1<2^{a-3}.
  \end{equation}
  For $q<10^3$, $\nu_2(\tfrac{1}{2}\ord_4(q)(q^2-1))\leq8$ holds, so \eqref{Neven} implies $q-1<32$ or $q\leq31$. For this range of primes, $\nu_2(\tfrac{1}{2}\ord_4(q)(q^2-1))\leq6$, so \eqref{Neven}
  further gives $q-1<8$ or $q\leq7$. This argument once again gives $\nu_2(\tfrac{1}{2}\ord_4(q)(q^2-1))\leq4$, so \eqref{Neven} implies $q-1<2$, a contradiction, concluding the proof.
 \end{proof}
 
 The number of prime pairs $(p,q)$ with $p<q<10^3$ is 14028. Propositions \ref{singleq}, \ref{poddq2}, \ref{p4q3}, \ref{p2q2}, \ref{p3q3}, \ref{pmq3}, \ref{p4q4} show that the only possible exceptions come
 from orders $N$ of the form $p^{2k}q^2$. The number of pairs with possible exceptions is only 162, and the total number of exceptions (that cannot be solved with the methods developed here) is 290.
 
 Indeed, we consider first the case $p=2$; then, either $N=2^{2k}q^2$ or $4q^{2k}$. Applying \eqref{lower} and \eqref{upper} in the latter case, we obtain
 \[2q^2\leq\sqrt{N}\leq\frac{4q}{2(1-\frac{1}{q})}.\]
 We remark that $b=1$ in \eqref{upper}, as no prime $q<10^3$ is Wieferich; furthermore, the quantity $F(N,n)$ in \eqref{upper} is $\leq F(N,2q)=4q$ due to Proposition \ref{Cpq}. 
 The above gives $q-1\leq 1$ a contradiction. Thus, all such exceptions are of the form $2^{2k}q^2$, $k\geq2$, hence applying \eqref{lower} and \eqref{upper} we obtain
 \[2^kq\leq\frac{2^aq}{2(1-\frac{1}{q})},\]
 which is equivalent to $2\leq k\leq a-1$, giving $a-2$ exceptions. We can easily calculate $a$ for every $q<10^3$ and verify\footnote{
 Check \texttt{orders2.wxmx} at \url{https://sites.google.com/site/romanosdiogenesmalikiosis/computational-data}} that the number of such exceptions for $N$ even are 240.
 
 If $N$ is odd, we have possible exceptions of the form $p^{2k}q^2$ or $p^2q^{2k}$, for $p<q<10^3$. Inequalities \eqref{lower} and \eqref{upper} give in the former case
 \[p^kq\leq\frac{p^aq^b}{4(1-\frac{1}{p})(1-\frac{1}{q})}<\frac{1}{2}p^aq^b.\]
 We either have $a=1$ or $b=1$ when $p,q<10^3$. For $a=1$ the above inequality can only hold for $b>1$, however, since $b\leq2$ if $a=1$, we get the pairs from \eqref{qa2}. The previous
 inequality becomes $2p^{k-1}<q$, and we have a total of 7 exceptions; every pair has an exception of the form $p^4q^2$ and there is also the additional exception $13^6 863^2$. When $b=1$, the inequality
 becomes $2p^k<p^a$ or $k<a$, or $a-2$ exceptions for every such pair as $k\geq2$. We can calculate all such exceptions\footnote{
 Check \texttt{orders.wxm} at \url{https://sites.google.com/site/romanosdiogenesmalikiosis/computational-data}}; they are another 42 in total which come from 32 pairs, and this tackles the case $N=p^{2k}q^2$ with
 $p<q<10^3$, which gives a total of 49 cases. When $N=p^2q^{2k}$, we have only one other possible exception, namely $13^2 239^4$; indeed, \eqref{lower} and \eqref{upper} give
 \[pq^k\leq\frac{p^aq^b}{4(1-\frac{1}{p})(1-\frac{1}{q})}<\frac{1}{2}p^aq^b.\]
 If $a=1$ we would have $b\leq2$ and $pq^k<pq^2$, a contradiction, as $k\geq2$. So $b=1$, so the above inequality becomes $2q^{k-1}<p^{a-1}$, which is only satisfied for $p=13$, $q=239$, $k=2$, thus
 obtaining a total of 50 possible exceptions when $N$ is odd. Therefore, the total number of possible exceptions is 290.

 \bigskip
 \bigskip
 
  \bibliographystyle{amsplain}
\bibliography{mybib}
 
 \end{document}